\documentclass[letterpaper, reqno,10 pt]{amsart}
\usepackage[margin=1.0in]{geometry}
\usepackage{amsmath,amssymb,amsthm}
\usepackage{mathrsfs}
\usepackage{etoolbox}
\usepackage{mathtools}

\makeatletter
\patchcmd{\subsubsection}{\scshape}{\bf}{}{}
\makeatother

\makeatletter
\renewcommand{\tocsection}[3]{%
  \indentlabel{\@ifnotempty{#2}{\bfseries\ignorespaces#1 #2\quad}}\bfseries#3}
\renewcommand{\tocsubsection}[3]{%
  \indentlabel{\@ifnotempty{#2}{\ignorespaces#1 #2\quad}}#3}
\renewcommand{\tocsubsubsection}[3]{%
  \indentlabel{\@ifnotempty{#2}{\hspace{1.4cm}\ignorespaces#1 #2\quad}}#3}

\makeatletter
\pretocmd{\chapter}{\addtocontents{toc}{\protect\addvspace{15\p@}}}{}{}
\pretocmd{\section}{\addtocontents{toc}{\protect\addvspace{5\p@}}}{}{}

\newcommand\@dotsep{4.5}
\def\@tocline#1#2#3#4#5#6#7{\relax
  \ifnum #1>\c@tocdepth 
  \else
    \par \addpenalty\@secpenalty\addvspace{#2}%
    \begingroup \hyphenpenalty\@M
    \@ifempty{#4}{%
      \@tempdima\csname r@tocindent\number#1\endcsname\relax
    }{%
      \@tempdima#4\relax
    }%
    \parindent\z@ \leftskip#3\relax \advance\leftskip\@tempdima\relax
    \rightskip\@pnumwidth plus1em \parfillskip-\@pnumwidth
    #5\leavevmode\hskip-\@tempdima{#6}\nobreak
    \leaders\hbox{$\m@th\mkern \@dotsep mu\hbox{.}\mkern \@dotsep mu$}\hfill
    \nobreak
    \hbox to\@pnumwidth{\@tocpagenum{\ifnum#1=1\bfseries\fi#7}}\par
    \nobreak
    \endgroup
  \fi}
\AtBeginDocument{%
\expandafter\renewcommand\csname r@tocindent0\endcsname{0pt}
}
\def\l@subsection{\@tocline{2}{0pt}{2.5pc}{5pc}{}}
\makeatother

\usepackage{esint}
\usepackage{fancybox}
\usepackage{color}   

\usepackage{hyperref}
\hypersetup{
   colorlinks=true, 
    linktoc=all,     
    linkcolor=blue,  
}

\newcommand{\Addresses}{{

  \bigskip
  \footnotesize

	\medskip
	\textsc{Instituto de Ciencias Matem\'aticas (CSIC), C. Nicol\'as Cabrera, 13-15, 28049 Madrid, Spain}\par\nopagebreak
  \textit{E-mail address}: \qquad \texttt{angel\textunderscore castro@icmat.es},\quad  \texttt{dcg@icmat.es},\quad \texttt{daniel.lear@icmat.es}  }}

\newtheorem{thm}{Theorem}[section]
\newtheorem{cor}[thm]{Corollary}

\newtheorem*{prop*}{Proposition}

\newtheorem{lemma}[thm]{Lemma}

\newcommand{\R}{\mathbb{R}}

\newcommand{\Z}{\Bbb{Z}}
\newcommand{\N}{\Bbb{N}}

\newcommand{\pa}{\partial}
\newcommand{\Pm}{\mathbb{P}_m}

\newcommand{\psim}{\psi^{[m]}}

\begin{document}

\title[Global Existence for the Confined IPM Equation]{\textsc{Global Existence of Quasi-Stratified Solutions \\for the Confined IPM Equation}}



\author{\'Angel Castro, Diego C\'ordoba and Daniel Lear}

\date{\today}

\maketitle

\begin{abstract}
In this paper, we consider a confined physical scenario to prove global existence of smooth solutions with bounded density and finite energy for the inviscid incompressible porous media (IPM) equation. The result is proved using the stability of stratified solutions, combined with an additional structure of our initial perturbation, which allows us to get rid of the boundary terms in the energy estimates. \\
\end{abstract}


\tableofcontents

\section{Introduction}\label{Sec_1}
In this paper we study the global in time existence of smooth solutions with bounded density and finite energy of the (2D) Incompressible Porous Media equation in a strip domain $\Omega$. That is, we consider the following active scalar equation:
$$\partial_t \varrho+\textbf{u}\cdot\nabla\varrho=0$$
with a velocity field $\mathbf{u}$ satisfying the momentum equation given by Darcy's law:
\begin{equation}\label{def_u}
\frac{\mu}{\kappa}\mathbf{u}=-\nabla p-g(0,\varrho),
\end{equation}
where $(\mathbf{x},t)\in\Omega\times\R^{+}$, $\mathbf{u}=(u_1,u_2)$ is the incompressible velocity (i.e., $\nabla\cdot \mathbf{u}=0$), $p$ is the pressure, $\mu$ is the dynamic viscosity, $\kappa$ is the permeability of the isotropic medium, $g$ is the acceleration due to gravity and $\varrho$ corresponds to the density transported without diffusion by the fluid.

Due to the direction of gravity, the horizontal and the vertical coordinates play different roles.  Here we assume spatial periodicity in the horizontal space variable, says $\varrho(x+2\pi k,y,t)=\varrho(x,y,t)$ and similarly $p(x+2\pi k,y,t)=p(x,y,t)$.  Finally, as these equations are studied on a bounded domain, we assume that our physical domain is impermeable, which
is exactly satisfied if $\mathbf{u}$ satisfies the no-slip boundary condition
\begin{equation}\label{Boundary_Conditions}
\mathbf{u}\cdot \mathbf{n}=0\qquad \text{on $\partial\Omega$,}
\end{equation}
where $\mathbf{n}$ denotes the exterior normal vector.

In this work we will focus on the case in which the evolution problem is posed on a porous strip with width $2l$. That is, the domain is the two-dimensional flat strip $\Omega:=\mathbb{T}\times[-l,l]$ with $0<l<\infty$.

This problem is known as the \emph{confined} IPM equation. Without loss of generality we will assume from now on that $\mu=\kappa=g=l=1$. To summarize, we have the following system of equations in $\Omega$:
\begin{equation}\label{System_A}
\left\{
\begin{array}{rl}
\partial_t\varrho +\mathbf{u}\cdot\nabla \varrho &= 0  \\
\mathbf{u}&=-\nabla p -(0,\varrho)\\
\nabla\cdot\mathbf{u}&=0 \\
\end{array}
\right.
\end{equation}
with the boundary condition $\mathbf{u}\cdot \mathbf{n}=0$ on $\partial\Omega\equiv\{y=\pm 1\}$. In our case, this implies that $u_2|_{\partial\Omega}=0$. In our physical system where there is gravity and stratification ($\mathbf{u}=0$ and $\varrho\equiv\varrho(y)$ is a stationary solution), vertical movement may be penalized while horizontal movement is not. This opens up the possibility of treating the corresponding initial value problem from a perturbative point of view.
As in \cite{Elgindi}, this paper studies the solutions of (\ref{System_A}) in the perturbative regime near the stratified state $\Theta(y):=-y$ for a specific type of perturbations:
\begin{equation}\label{perturbation}
\varrho(x,y,t)=\Theta(y)+\rho(x,y,t) \qquad (\mathbf{x},t)\in\Omega\times\R^{+}.
\end{equation}
The main result is that small perturbations $\rho$ in a suitable Sobolev space $X^k(\Omega)$, which we define below in (\ref{Space_X}), converge to a shear and nearby stationary flow in the sense: $\varrho(x,y,t)\equiv \Theta(y)+\rho(x,y,t)\to \Theta(y)+\rho_{\infty}(y)$ and $\mathbf{u}(x,y,t)\to 0$ as $t\to \infty$. The main mechanism of decay can be seen from the linearized equation:
$$\partial_t\rho(x,y,t)=-\Theta'(y)\,u_2(x,y,t)$$
which, after solving the velocity $\mathbf{u}=(u_1,u_2)$ in terms of $\rho$ yields
$$\partial_t\rho(x,y,t)=\Theta'(y) \left(\rho(x,y,t)+(-\Delta_{\Omega})^{-1}\partial_y^2\rho(x,y,t) \right).$$
Setting $\Theta(y):=-y$, the previous equation clearly shows the frequency dependent exponential decay over time of $\rho$, except the zero mode in $x$.
The goal of the present paper is to show how to control the nonlinearity, so that it would not destroy the decay provided by the linearized equation.

To do this, controlling the boundary terms  is the new additional difficulty. This  can be done by working with perturbations in the appropriate Sobolev space $X^{k}(\Omega)$. Using standard techniques, we will prove local in time existence of solutions for the perturbated problem in the space $X^k(\Omega)$. For sake of completeness we include the proof, where the cornerstone will be the properties of an orthonormal basis adapted to $X^k(\Omega)$. The reason for working with initial perturbations with that additional structure will be seen in the apriori energy estimates. There, all the boundary terms that \hyphenation{appear} appear in the computations vanish thanks to \hyphenation {pe-rio-di-ci-ty} periodicity in the horizontal variable and by the additional structure of our initial perturbations, which is preserved in time by the local existence result, as long as the solution exists.\\

Namely, we will prove the following result:

\noindent
\textbf{Theorem}
\emph{The stratified state $\Theta$ of the confined IPM equation is asymptotically stable in $X^{\kappa}(\Omega)$ for $\kappa\geq 10$. In other words, there exists $\varepsilon_0>0$ such that if we solve (\ref{System_A}) with initial data $\varrho(0)=\Theta+\rho(0)$ and $\rho(0)\in X^{\kappa}(\Omega)$ with $||\rho||_{H^{\kappa}(\Omega)}(0)\leq \epsilon_0$ then, the solution exists globally in time and satisfies:
\begin{enumerate}
	\item $||\mathbf{u}||_{H^{3}(\Omega)}(t)\lesssim \epsilon_0 \,(1+t)^{-\frac{5}{4}}$,
	\item $||\bar{\varrho}||_{H^{3}(\Omega)}(t)\lesssim \epsilon_0\,(1+t)^{-\frac{5}{4}}$,
	\item $||\tilde{\varrho}-\Theta||_{H^{\kappa}(\Omega)}(t)\leq 2\,\varepsilon$.
\end{enumerate}
where $\varrho:=\bar{\varrho}+\tilde{\varrho}$ such that $\bar{\varrho}\perp\tilde{\varrho}$ and $\bar{\varrho}$ is given by the projection operator onto the subspace of functions with zero average in the horizontal variable.\\
}

\noindent
\textbf{Remark:} If we perturb the stratified state by a function of $y$ only then there should be no decay. For this reason, the orthogonal decomposition $\varrho=\bar{\varrho}+\tilde{\varrho}$ will be considered.
\\

\noindent
\textbf{Remark:} The strategy used in our paper can be applied to a more general class of monotone shear flows.  The  proof works for small perturbations in some sense of our steady state with $\Theta'<0$. However, a highly non trivial problem is to extend this to the case of possibly degenerate shear flows where $\Theta'=0$ at some value.\\


A more precise statement of our result is presented as Theorem (\ref{main_thm}), where we
also illustrate its proof through a bootstrap argument. Despite the apparent simplicity, understanding the stability of this flow is far from been trivial.

\subsection{Motivation}
The study of partial differential equations arising in fluid mechanics has been an
active field in the past century, but many important and physically relevant questions  remain  wide  open  from  the  point  of  view  of  mathematical  analysis. Among the problems that attracted recently renewed interest, \textit{active scalar} equations that arise in fluid dynamics present a challenging set of problems in PDE. Maybe
the best example is the Surface Quasi-Geostrophic equation (SQG), introduced in the mathematical literature
in \cite{Constantin-Majda-Tabak}. The inviscid SQG equation in $\R^2$ takes the form:
\begin{equation*}
\begin{cases}
\partial_t \theta +\mathbf{u}\cdot\nabla \theta=0\\
\mathbf{u}=\mathbf{R}^{\perp}\theta\\
\end{cases}
\end{equation*}
where $\mathbf{R}=(R_1,R_2)$ denote the 2D Riesz transforms. This problem has been widely investigated due to their mathematical analogies with the 3D Euler equation, but little is known. Local well-posedness and
regularity criteria in various functional settings have been established, see  \cite{Cordoba-Castro-Gomez-Serrano} as a survey. The global regularity problem for the Cauchy problem with a general smooth initial data remains open. Besides radially symmetric solutions, which are all stationary, the first examples of non-trivial global smooth solutions we are aware of were recently provided in \cite{Cordoba-Castro-Gomez-Serrano}. An alternative construction of smooth families of global special solutions can be found in \cite{Gravejat-Smets}, where the authors focus on travelling-wave solutions to the inviscid SQG. On the other hand, whether finite time blow up  can happen for smooth initial data remains completely open.

It is important to note that, both IPM and SQG, the operator relating the velocity and the active scalar is a \textit{singular integral operator} of  zero order. Even more, in the whole space, the velocity (\ref{def_u}) can be rewritten in a more convenient way as $\mathbf{u}=\mathbf{R}^{\perp}R_1\varrho$. Despite the fact that there are great similarities between the inviscid versions of SQG and IPM equations, there are also important differences.
This work appears to be the first to find an scenario to prove global existence  of smooth solutions with bounded density and finite energy for the inviscid IPM equation.

\subsubsection{The question of long-time behavior}
A fundamental challenge in mathematical physics is to understand the dynamics  of  physical  systems  as  they  evolve  over long times. This is particularly true when it comes to the study of the long-time behavior of such systems without dissipation. Depending upon the specific physical situation that a given fluid equation models, we find vastly different mathematical objects arising. In recent years, researchers have discovered numerous interesting phenomena such as the existence of solutions whose long-time behavior is determined entirely by:
\begin{itemize}
	\item some linear or dispersive effect, for example in water waves \cite{Germain-Masmoudi-Shatah}, \cite{Ionescu-Pusateri} and \cite{Wu}.
	\item some linear mixing effect, for the Couette flow in Navier-Stokes and Euler equations \cite{Bedrossian-Germain-Masmoudi}, \cite{Bedrossian-Masmoudi}.
	\item some hypocoercive dissipative mechanisms, for kinetic theory \cite{Desvillettes-Villani_1} and \cite{Desvillettes-Villani_2}.
\end{itemize}

The idea of taking a non-linear equation where global well-posedness is unknown and to prove it for a perturbation ``close'' to a stationary solution of the equation is so natural. For small enough initial data, one might conjecture that solutions to the nonlinear problem behave asymptotically like solutions of the corresponding linear problem.

As in \cite{Elgindi}, where the author gives in $\R^2$ the first construction of a non-trivial global smooth solution for the inviscid IPM equation, the main idea is that stratification can be a stabilizing force. One can imagine that a fluid with density that is proportional to depth is in some sense ``stable''. The mechanism behind the stability is that the linearized IPM equation around the stratified state exhibit certain damping properties.  This convergence back to equilibrium, despite the lack of dissipative mecha\-nisms, is known as \textit{inviscid damping} and is a close relative of Landau damping in plasma physics.  It was proved that Landau damping provides a similar stability for Vlasov-Poisson in Mouhot and Villani's breakthrough work \cite{Mouhot-Villani}.

\subsubsection{Previous results for IPM with smooth initial data}
In {\cite{Cordoba-Gancedo-Orive}, the local existence and uniqueness in H\"older space $C^{\delta}$ with $0<\delta<1$ was shown by the particle-trajectory method for the whole space case. By a similar approach, the local well-posedness  in Besov and Triebel-Lizorkin spaces was proved in \cite{Xue} and \cite{Yu-He}.

For the Lagrangian formulation, in \cite{Constantin-Vicol-Wu}, the  authors show that as long as the solution of this equation is in a class of regularity that assures H\"older continuous gradients of the velocity, the corresponding Lagrangian paths are real analytic functions
of time.

In the class of weaker solutions, the results of \cite{Cordoba-Faraco-Gancedo} and \cite{Shvydkoy} establish the non-uniqueness of $L_{t,x}^{\infty}$ weak solutions to the inviscid IPM equation starting from the zero solution. Recently, in \cite{Isett-Vicol} the authors were able to construct global weak solutions to the inviscid IPM equation which are of class $C_{t,x}^{\delta}$ with $\delta<1/9$ starting from a smooth initial data. All these works are based on a variant of the method of convex integration.

In the direction of classical solutions, the only result known is due to Elgindi \cite{Elgindi} shows that solutions which are ``close'' to certain stable stratified solutions exists globally in time, but since he works in the whole space, such solutions have unbounded density. He considers perturbations in two settings which are fundamentally different:
\begin{itemize}
	\item On the whole space $\R^2$: In this case the stationary solution does not belong to $L^2(\R^2)$. However, the author can perturb the stationary solution by a sufficiently small $H^s$ function, and to prove that the perturbation decay to equilibrium as $t\to +\infty$.

	\item On the two dimensional torus $\mathbb{T}^2$: Similarly, the stationary solution is not periodic but the author may perturb it by a periodic function and once more the perturbation will remain periodic. The result here is quite different for the main reason that $\varrho$ itself does not decay. Even so, smooth perturbations of the stationary solution are stable for all time in Sobolev spaces.
\end{itemize}

\noindent
We now motivate our attack setting. We start with the observation that gravity term in Darcy's law (\ref{def_u}) convert IPM in an anisotropic problem, which implies different properties in different directions. In our case, the vertical direction pointing in the direction of gravity will play a key role. By this anisotropic property, it seems natural that  $\mathbb{T}\times[-1,1]$ might be an adequate scenario to set our equations.\\

In order to solve our problem in the bounded domain $\Omega$, in certain Sobolev space, we have to overcome the following new difficulties:
\begin{enumerate}
	\item[i)] To be able to handle the boundary terms that appear in the computations.
	
	\item[ii)] The lack of higher order boundary conditions at the boundaries, due to the fact that we work in Sobolev spaces.
\end{enumerate}

\noindent
Indeed, both difficulties i) and ii) can be bypass if our initial perturbation has a special structure. We introduce the following spaces to characterize our initial data:
\begin{align}
X^k(\Omega)&:=\{f\in H^{k}(\Omega):\partial_y^{n}f|_{\partial\Omega}=0 \quad \text{for } n=0,2,4,\ldots,k^{\star}\}\label{Space_X},\\
Y^k(\Omega)&:=\{f\in H^{k}(\Omega):\partial_y^{n}f|_{\partial\Omega}=0 \quad \text{for } n=1,3,5,\ldots,k_{\star}\}\label{Space_Y}
\end{align}
where, we defined the auxiliary values of $k^{\star}$ and $k_{\star}$ as follows:
$$k^{\star}:= \begin{cases}
k-2 \qquad k \quad \text{even}\\
k-1 \qquad k \quad \text{odd}
\end{cases} \qquad \text{and }  \qquad k_{\star}:= \begin{cases}
k-1 \qquad k \quad \text{even}\,\\
k-2 \qquad k \quad \text{odd}.
\end{cases} $$
Lastly, we remember that the Trace operator $T:H^{1}(\Omega)\rightarrow L^2(\partial \Omega)$ defined by $T[f]:=f|_{\partial \Omega}$ is bounded for all $f\in H^{1}(\Omega)$. Consequently, both spaces are well defined.

\subsection{The equations}
In this section, we describe the equation that a perturbation of the stratified solution (\ref{perturbation}) must satisfy. In order to prove our goal, we plug into the system \eqref{System_A} the following ansatz:
\begin{align*}
\varrho(x,y,t)&= -y+\rho(x,y,t),\\
p(x,y,t)&= \Pi(x,y,t) -\tfrac{1}{2}y^2+\int_{0}^y \tilde{\rho}(y',t)dy'
\end{align*}
where, for a general function $f:\Omega\times\R^{+}\rightarrow \R$, we define
\begin{equation*}
\tilde{f}(y,t):=\frac{1}{2\pi}\int_{-\pi}^\pi f(x',y,t)dx'\quad \text{and} \quad \bar{f}(x,y,t):=f(x,y,t)-\tilde{f}(y,t).
\end{equation*}

Then, for the perturbation $\rho$, we obtain the system
\begin{equation}\label{System_A_PI}
\left\{
\begin{array}{rl}
\partial_t\rho +\mathbf{u}\cdot\nabla \rho &=  u_2  \\
\mathbf{u}&=-\nabla \Pi -(0,\bar{\rho})\\
\nabla\cdot\mathbf{u}&=0
\end{array}
\right.
\end{equation}
besides the boundary condition $\mathbf{u}\cdot \mathbf{n}=0$ on $\partial\Omega$. Note that in $\Omega$, our perturbation $\rho$ does not have to decay in time. Indeed, if we perturb the stationary solution  by a function of $y$ only there is no decay. More specifically, $\rho\equiv\rho(y)$ and $u=0$ is a stationary solution of \eqref{System_A_PI}. To overcome this difficulty, the orthogonal decomposition $\rho= \bar{\rho}+\tilde{\rho}$ will be considered.

The system (\ref{System_A_PI}) can be rewritten in terms of $\bar{\rho}$ and $\tilde{\rho}$ as follows:
\begin{equation}\label{System_B}
\left\{
\begin{array}{rl}
\partial_t \bar{\rho}+\overline{\textbf{u}\cdot \nabla \bar{\rho}}+\partial_y\tilde{\rho} \,u_2&=u_2\\
\partial_t \tilde{\rho}+ \widetilde{\textbf{u}\cdot \nabla \bar{\rho}}\hspace{1.4 cm }&=0\\
\textbf{u}&=-\nabla\Pi-(0,\bar{\rho})\\
\nabla\cdot\mathbf{u}&=0
\end{array}
\right.
\end{equation}
Notice that $\tilde{\rho}$ is always a function of $y$ only and $\bar{\rho}$ has zero average in the horizontal variable. It is expected that $\bar{\rho}$ will decay on time and $\tilde{\rho}$ will just remain bounded.
The systems (\ref{System_A_PI}) and (\ref{System_B}) are the same, but depending on what we need, we will work with one or the other.

\subsection{Notation \& Organization}
We shall denote by $(f,g)$ the $L^2(\Omega)$ inner product of $f$ and $g$.
As usual, we use bold for vectors valued functions. Let $\mathbf{u}=(u_1,u_2)$ and $\mathbf{v}=(v_1,v_2)$, we define
$\langle \mathbf{u},\mathbf{v}\rangle=(u_1,v_1)+(u_2,v_2)$. Also, we remember that the natural norm in Sobolev spaces is defined by:
$$||f||_{H^{k}(\Omega)}^2:=||f||_{L^2(\Omega)}^2+||f||_{\dot{H}^{k}(\Omega)}^2, \qquad ||f||_{\dot{H}^{k}(\Omega)}^2:=||\partial^{k} f||_{L^2(\Omega)}^2.$$
For convenience, in some place of this paper, we may use $L^2,\dot{H}^k $ and $H^k$ to stand for $L^2(\Omega), \dot{H}^k(\Omega)$ and $H^{k}(\Omega)$, respectively. Moreover, to avoid clutter in computations, function arguments (time and space) will be omitted whenever they are obvious from context. Finally, we use the notation $f\lesssim g$ when there exists a constant $C>0$ independent of the parameters of interest such that $f\leq Cg$.\\

\noindent
\textbf{Organization of the paper:}
In Section \ref{Sec_2}, we introduce the functional spaces $X^k(\Omega)$ and $Y^k(\Omega)$ where we will work. The key point of working with initial perturbations with the structure given by these spaces it is showed in Section \ref{Sec_3}. Section \ref{Sec_4} contains the proof of the local existence in time for initial data  in $X^k(\Omega)$ for the \textit{confined} problem, together with a blow-up criterion. The core of the article is the proof of the main theorem in Section \ref{Sec_5}. We commence by the \textit{a priori} energy estimates given in Section \ref{Sec_5.1}. This is followed by an explanation of the decay given by the linear semigroup of our system in Section \ref{Sec_5.2}. Finally, in Section \ref{Sec_5.3} we exploit a bootstrapping argument to prove our theorem.

\section{Mathematical setting and preliminares}\label{Sec_2}

In this section, we will see the importance of working with initial perturbations belonging to $X^k(\Omega)$.
We also consider an  adapted orthonormal basis for working with these perturbations, together with their eigenfunction expansion.

\subsection{Motivation of the spaces $X^k(\Omega)$ and $Y^k(\Omega)$.}

By the no-slip condition $u_2(t)|_{\partial\Omega}=0$, the solution $\rho(t)$ of $(\ref{System_A_PI})$ satisfies  the following transport equation on the boundary:
\begin{equation}\label{transport_eq}
\partial_t\rho(t)|_{\partial\Omega}+u_1(t)\partial_x\rho(t)|_{\partial\Omega}=0.
\end{equation}
As our objective is to obtain global stability and decay to equilibrium of sufficiently small perturbations, it seems natural to consider $\rho(0)|_{\partial\Omega}=0$. Then, by the transport character of (\ref{transport_eq}) the initial condition is preserved in time $\rho(t)|_{\partial\Omega}=0$ as long as the solution exists. In addition, taking derivatives in Darcy's law, using the incompressibility condition, and restricting to the boundary we have
\begin{align}\label{anulacion}
\partial_y u_1(t)|_{\partial\Omega}=0\quad\text{and}\quad
\partial_y^{2}u_2(t)|_{\partial\Omega}=0,
\end{align}
given that $\rho(t)|_{\partial\Omega}=u_2(t)|_{\partial\Omega}=0$. Relations \eqref{anulacion}  give rise  to the following equation for the  derivative in time of  $\partial^2_y\rho(t)$ at the boundary:
\begin{align*}
\partial_t\partial_y^{2}\rho(t)|_{\partial\Omega}= -u_1(t)\partial_x(\partial_y^2\rho)(t)|_{\partial\Omega}-\partial_y u_2(t) \partial^2_y\rho(t)|_{\partial\Omega}.
\end{align*}
Thus, we find that $\partial^2_y\rho(0)|_{\partial\Omega}=0$ implies that $\partial_t\partial^2_y\rho(t)|_{\partial\Omega}=0$, and consequently the condition on the boundary is preserved in time.

Iterating this procedure we can check that  the conditions $\partial_y^{n}\rho(0)|_{\partial\Omega}=0$,  for $n=2,4,...$ are preserved  in time. This is the reason why we can look for solutions $\rho(t)$ in the space $X^k(\Omega)$, if the initial data belongs to it. Moreover $u_1(t)$ will belong to $Y^k(\Omega)$ and $u_2(t)$ will belong to $X^k(\Omega)$.

\subsection{Biot-Savart law and stream formulation}
In the whole space $\R^2$ we have a simple expression for $\nabla\Pi$ in terms of $\bar{\rho}$: $$\nabla\Pi=\nabla(-\Delta)^{-1}\partial_y\bar{\rho}$$ so, we can write the velocity in terms of $\bar{\rho}$ as
$$\mathbf{u}=-\nabla \Pi-(0,\bar{\rho})=\mathbf{R}^{\perp}R_1\bar{\rho}$$
where $\mathbf{R}^{\perp}=(-R_2,R_1)$, being $R_i$ the Riesz's transform.

In our setting $\Omega=\mathbb{T}\times[-1,1]$, to obtain an analogous expression we proceed as follow. Due to the incompressibility of the flow, by taking the divergence of Darcy's law we find that
\begin{equation}\label{Delta_PI}
\Delta \Pi= -\partial_y\bar{\rho}.
\end{equation}
Moreover, the no-slip conditon (\ref{Boundary_Conditions}) give us the following boundary condition:
\begin{equation}\label{PI_boundary}
\partial_y \Pi|_{\partial\Omega}=-\bar{\rho}|_{\partial\Omega}=0,
\end{equation}
which vanishes as  $\rho\in X^k(\Omega)$. Then, putting together (\ref{Delta_PI}) and (\ref{PI_boundary}) (notice that we look for a periodic in the $x$-variable $\Pi$), we recover the velocity field, in terms of $\bar{\rho}$, by the expression $\mathbf{u}=-\nabla\Pi-(0,\bar{\rho})$.\\

Other way to reach this expression follows the next steps: as $\nabla\cdot\mathbf{u}=0$,  we can write the velocity as the gradient perpendicular of a \textit{stream function} $\psi$, i.e.,
\begin{equation}\label{Stream_1}
\mathbf{u}=\nabla^{\perp}\psi
\end{equation}
with $\nabla^{\perp}\equiv (-\partial_y,\partial_x)$.  Then, applying the   $\emph{curl}$ operator on (\ref{def_u}), we get the Poisson equation for $\psi$:
\begin{equation*}
\Delta\psi=-\partial_x\bar{\rho}.
\end{equation*}
  Taking in account (\ref{Stream_1}) and the no-slip condition (\ref{Boundary_Conditions}) we obtain  the boundary condition:
\begin{equation*}
\partial_x\psi|_{\partial\Omega}=0.
\end{equation*}
Thus, we need to impose $\psi|_{\{y=\pm 1\}}=c_{\pm}$ where $c_{+}$ could be, in principle, different from $c_{-}$. However, the periodicity in the $x$-variable of $\Pi$ forces to take $c_{+}=c_{-}$, and since we are only interested in the derivatives of $\psi$ we will take $c_{\pm}=0$.\\

To sum up, in order to close the system of equations,  we first solve either
\begin{equation*}
\left\{
\begin{array}{rlrr}
\Delta \Pi&= -\partial_y\bar{\rho}   \qquad &\text{in}\quad &\Omega,\\
\partial_y \Pi&=0  \hspace{1.1 cm} &\text{on} \quad &\partial\Omega,
\end{array}
\right.
\end{equation*}
or
\begin{equation} \label{Poisson_Homogeneous_Neumann_psi}
\left\{
\begin{array}{rlrr}
\Delta \psi&= -\partial_x\bar{\rho}  \qquad &\text{in}\quad &\Omega,\\
\psi &=0 \hspace{1.1 cm} &\text{on}  \quad &\partial\Omega,
\end{array}
\right.
\end{equation}
and after that write
\begin{equation*}
\mathbf{u}=-\nabla \Pi-(0,\bar{\rho}) \qquad \text{or} \qquad \mathbf{u}=\nabla^{\perp}\psi.
\end{equation*}

In the rest of the paper we will use the \textit{stream formulation} to recover the velocity field. In the next section,  we present an orthonormal basis of $X^k(\Omega)$ in order to solve (\ref{Poisson_Homogeneous_Neumann_psi}), which allows to write the velocity in terms of the ``Fourier coefficients" of $\bar{\rho}$.

\subsection{An orthonormal  basis for $X^k(\Omega)$}
Our goal is to solve (\ref{Poisson_Homogeneous_Neumann_psi}). In order to do this, we define:
$$a_p(x):=\frac{1}{\sqrt{2\,\pi}}\, \exp\left( i p x  \right) \quad \text{with } x\in \mathbb{T} \quad \text{for } p\in \Z$$
and
$$b_{q}(y):= \begin{cases}
\cos\left(q y \frac{\pi}{2}  \right) \qquad q \quad \text{odd}\\
\sin\left(q y \frac{\pi}{2}   \right) \qquad q \quad \text{even}
\end{cases} \quad \text{with } y\in[-1,1] \quad \text{for } q\in \N $$
where $\{a_p\}_{p\in\Z}$ and $\{b_{q}\}_{q\in\N}$ are orthonormal basis for $L^2(\mathbb{T})$ and $L^2([-1,1])$ respectively. Indeed, $\{b_{q}\}_{q\in\N}$ consists of eigenfunctions of the operator $S=(1-\pa^2_y)$ with domain $\mathscr{D}(S)=\{f\in H^2[-1,1]\,:\, f(\pm 1)=0\}$.    Consequently, the product of them $\omega_{p,q}(x,y):=a_{p}(x)\,b_{q}(y)$ with $(p,q)\in\Z\times\N$ is an orthonormal basis for the product space $L^2(\mathbb{T}\times[-1,1])\equiv L^2(\Omega)$.

Now, we define an  auxiliary orthonormal basis for $L^2([-1,1])$ given by
$$c_{q}(y):= \begin{cases}
\sin\left(q y \frac{\pi}{2}  \right) \qquad q \quad \text{odd}\\
\cos\left(q y \frac{\pi}{2}   \right) \qquad q \quad \text{even}
\end{cases} \quad \text{with } y\in[-1,1] \quad \text{for } q\in \N\cup\{0\},$$
consisting of eigenfunctions of the operator $S=1-\pa^2_y$ with domain $\mathscr{D}(S)=\{f\in H^2[-1,1]\,:\, (\pa_yf)(\pm 1)=\nolinebreak0\}$.
In the same way as before, the product $\varpi_{p,q}(x,y):=a_{p}(x)\,c_{q}(y)$ with $(p,q)\in\Z\times\left(\N\cup \{0\}\right)$ is again an orthonormal basis for $L^2(\Omega)$.\\

\noindent
\textbf{Remark:} Let us describe the analogue of the Fourier expansion in terms of our eigenfunctions expansion. This is, for $f\in L^2(\Omega)$, we have the $L^2(\Omega)$-conergence given by:
\begin{equation}\label{def_eigenfunction_expansion}
f(x,y)=\sum_{p\in\Z}\,\sum_{q\in \N}\mathcal{F}_{\omega}[f](p,q)\,\omega_{p,q}(x,y) \quad \text{where} \quad \mathcal{F}_{\omega}[f](p,q):=\int_{\Omega}f(x',y')\,\overline{\omega_{p,q}(x',y')}\,dx'dy'
\end{equation}
or
\begin{equation}\label{def_eigenfunction_expansion_base_Y}
f(x,y)=\sum_{p\in\Z}\,\sum_{q\in \N\cup \{0\}}\mathcal{F}_{\varpi}[f](p,q)\,\varpi_{p,q}(x,y) \quad \text{where} \quad \mathcal{F}_{\varpi}[f](p,q):=\int_{\Omega}f(x',y')\,\overline{\varpi_{p,q}(x',y')}\,dx'dy'.
\end{equation}

The main result of this part is to see that $\{\omega_{p,q}\}_{(p,q)\in\Z\times\N}$ is an orthonormal basis not only for $L^2(\Omega)$ but for $X^k(\Omega)$, and that $\{\varpi_{p,q}\}_{(p,q)\in\Z\times(\N\cup\{0\})}$ is basis of $Y^k(\Omega)$. The sequence $\{a_p\}_{p\in\Z}$ is the standard Fourier basis in $H^k(\mathbb{T})$. Then, we will  focus only on the convergence properties of \emph{span}$\{b_1,b_2,b_3,\ldots\}$ and \emph{span}$\{c_0,c_1,c_2,\ldots\}$.\\

As we will see below, the relation between derivatives of $\{b_q\}_{q\in\N}$ and $\{c_q\}_{q\in\N\cup\{0\}}$ plays a key role in the convergence properties. An easy computation gives us:
\begin{equation}\label{pa_b}
(\pa_y b_q)(y)=(-1)^{q} q \tfrac{\pi}{2} c_q(y) \qquad \text{for} \quad q\in\N
\end{equation}
and
\begin{equation}\label{pa_c}
(\pa_y c_q)(y)=\begin{cases}
-(-1)^{q} q \tfrac{\pi}{2} b_q(y)   \quad &q\in\N,\\
0   \quad &q=0.
\end{cases}
\end{equation}
Then, as a consequence of (\ref{pa_b}) and (\ref{pa_c}), for $q\in\N$ we have:
\begin{equation}\label{pa^2_basis}
(\pa_y^2 b_q)(y)=- \left(q \tfrac{\pi}{2}\right)^2 b_q(y) \qquad \text{and} \qquad  (\pa_y^2 c_q)(y)=- \left(q \tfrac{\pi}{2}\right)^2 c_q(y).\\
\end{equation}

\noindent
Hence, for each $f\in L^2([-1,1])$, as $\{b_q\}_{q\in\N}$ and $\{c_q\}_{q\in\N\cup\{0\}}$ are orthonormal bases for $L^2([-1,1])$ we have:
\begin{equation}\label{L^2[-L,L]_CONVERGENCE}
P_M f\xrightarrow{M\rightarrow \infty} f \qquad \text{and} \qquad Q_M f\xrightarrow{M\rightarrow \infty} f \qquad \text{in} \quad L^2([-1,1])
\end{equation}
where the partial sums are given by:
\begin{equation}\label{Partial_sum}
P_M f(y)=\sum_{m=1}^{M} \left\langle f,b_m \right\rangle\, b_m(y) \qquad \text{and} \qquad  Q_M f(y)=\sum_{m=0}^{M} \left\langle f,c_m\right\rangle\, c_m(y).
\end{equation}
\textbf{Remark:} Here, the notation $\left\langle\cdot,\cdot\right\rangle$ refers to the inner product in $L^2([-1,1])$.\\

We are now ready to present the main lemmas of this section. Let us recall first definitions (\ref{Space_X}) and (\ref{Space_Y}), which give us:
$$X^k([-1,1])=\{f\in H^{k}([-1,1]): (\partial_y^{n}f)(\pm 1)=0 \quad \text{for } n=0,2,4,\ldots,k^{\star}\}\,$$
and
$$Y^k([-1,1])=\{f\in H^{k}([-1,1]): (\partial_y^{n}f)(\pm 1)=0 \quad \text{for } n=1,3,4,\ldots,k_{\star}\}.$$

\begin{lemma}\label{Base_X^k}
$\{b_q\}_{q\in\N}$ is an orthonormal base of $X^k([-1,1])$.
\end{lemma}
\begin{proof}
Since the orthogonality is trivial, we will give the details of the completeness of the basis. For a function $f\in X^{k}([-1,1])$ we know that $f\in H^{k}([-1,1])$. Then, by (\ref{L^2[-L,L]_CONVERGENCE}) we have that:
$$P_m \pa_y^n f \xrightarrow{M\rightarrow \infty} \pa_y^n f  \qquad \text{in} \quad  L^2([-1,1]) \qquad \text{for} \quad  n=0,2,4,\ldots,\text{either $k$ or $k-1$}.$$
By (\ref{Partial_sum}) we get:
\begin{equation}\label{P_M derivadas}
P_M \partial_y^n f=\sum_{m=1}^{M}\left\langle \pa_y^n f,b_m \right\rangle\, b_m(y)
\end{equation}
where, by integration by parts and (\ref{pa^2_basis}), we have:
\begin{align}\label{recusion_derivadas_1}
\left\langle \pa_y^n f,b_m \right\rangle=\int_{-1}^{+1}\pa_y^n f(y') b_m(y')\,dy'&=\int_{-1}^{+1}f(y') \pa_y^n b_q(y')\,dy' \nonumber\\
&=(-1)^n \left( q \tfrac{\pi}{2}\right)^{n}\int_{-1}^{+1}f(y') b_q(y')\,dy' \nonumber\\
&=(-1)^n \left( q \tfrac{\pi}{2}\right)^{n} \left\langle  f,b_m \right\rangle.
\end{align}
We must note that, thanks to $b_q(\pm 1)=0$ and the boundary conditions, the boundary terms in the integration by parts vanish. Therefore, putting (\ref{recusion_derivadas_1}) in (\ref{P_M derivadas}) and applaying again (\ref{pa^2_basis}) we arrive to $P_M \pa_y^n f\equiv \pa_y^n P_M f$ and we obtain:
$$\pa_y^n P_M  f \xrightarrow{M\rightarrow \infty} \pa_y^n f  \qquad \text{in} \quad  L^2([-1,1]) \qquad \text{for} \quad  n=0,2,4,\ldots,\text{either $k$ or $k-1$}.$$
Moreover, by (\ref{L^2[-L,L]_CONVERGENCE}) we have:
$$Q_M \pa_y^{n+1} f \xrightarrow{M\rightarrow \infty} \pa_y^{n+1} f  \qquad \text{in} \quad  L^2([-1,1]) \qquad \text{for} \quad  n=0,2,4,\ldots,k^*$$
where by (\ref{Partial_sum}), we get:
\begin{equation}\label{Q_M derivadas}
Q_M \partial_y^{n+1} f=\sum_{m=0}^{M}\left\langle \pa_y^{n+1} f,c_m \right\rangle\, c_m(y).
\end{equation}
We notice that $\left\langle \pa_y^{n+1} f,c_0 \right\rangle=0$ due to the fact that $(\pa_y^n f)(\pm 1)=0$ by hypothesis. In addition, by integration by parts and (\ref{pa_c}) for $m\geq 1$ we obtain:
\begin{align}\label{recusion_derivadas_2}
\left\langle \pa_y^{n+1} f,c_m \right\rangle=\int_{-1}^{+1}\pa_y^{n+1} f(y') c_q(y')\,dy'&=-\int_{-1}^{+1}\pa_y^{n} f(y') (\pa_y c_q)(y')\,dy' \nonumber\\
&=(-1)^q \left(q\tfrac{\pi}{2}\right) \int_{-1}^{+1} \pa_y^n f(y') b_q(y')\, dy'\nonumber\\
&=(-1)^q \left( q\tfrac{\pi}{2}\right) \left\langle \pa_y^{n} f,b_m \right\rangle.
\end{align}
Here, the boundary term vanishes because by hypothesis we have that $(\pa_y^n f)(\pm 1)=0$.
Hence, putting (\ref{recusion_derivadas_2}) in (\ref{Q_M derivadas}) and applaying again (\ref{pa_b}) we arrive to $Q_M \partial_y^{n+1} f\equiv \pa_y P_M \pa_y^n f\equiv \pa_y^{n+1}P_M f$. Therefore:
$$\pa_y^{n+1} P_M  f \xrightarrow{M\rightarrow \infty} \pa_y^{n+1} f  \qquad \text{in} \quad  L^2([-1,1]) \qquad \text{for} \quad  n=0,2,4,\ldots,k^{\star}.$$
\end{proof}

\begin{lemma}\label{Base_Y^k}
$\{c_q\}_{q\in\N\cup\{0\}}$ is an orthonormal base of $Y^k([-1,1])$.
\end{lemma}
\begin{proof}
This results follows from the same ideas than the proof of the above Lemma (\ref{Base_X^k}).
\end{proof}

Because of  Lemmas \ref{Base_X^k} and \ref{Base_Y^k} one has the following expressions for both the $X^k(\Omega)$ and $Y^k(\Omega)$ norms.

\begin{cor}\label{thm_norma_expansion}
Let $f\in X^k(\Omega)$ and $g\in Y^k(\Omega)$. For $s_1,s_2\in\N\cup\{0\}$ such that $s_1+s_2\leq k$, we have:
\begin{align*}
||\partial_x^{s_1}\partial_y^{s_2}f||_{L^2(\Omega)}^2&=\sum_{p\in\Z}\sum_{q\in\N}|p|^{2 s_1} |q\tfrac{\pi}{2}|^{2 s_2}\left|\mathcal{F}_{\omega}[f](p,q)\right|^2\\
||\partial_x^{s_1}\partial_y^{s_2}g||_{L^2(\Omega)}^2&=\sum_{p\in\Z}\sum_{q\in\N\cup\{0\}}|p|^{2 s_1} |q\tfrac{\pi}{2}|^{2 s_2}\left|\mathcal{F}_{\varpi}[f](p,q)\right|^2
\end{align*}
where $\mathcal{F}_{\omega}[f](p,q)$ and $\mathcal{F}_{\varpi}[f](p,q)$ are given by (\ref{def_eigenfunction_expansion}) and (\ref{def_eigenfunction_expansion_base_Y}) respectively.\\
\end{cor}

Introducing a threshold number $m\in\N$, we define the projections $\mathbb{P}_m$ and $\mathbb{Q}_m$ of $L^2(\Omega)$ onto the linear span of eigenfunctions generated by $\{\omega_{p,q}\}_{(p,q)\in \Z\times \N}$ and $\{\varpi_{p,q}\}_{(p,q)\in \Z\times \N\cup\{0\}}$  respectively, such that $\{|p|,\,q\} \leq m$. This is, we have that:
\begin{equation}\label{projector_def}
\mathbb{P}_m [f](x,y):=\sum_{\substack{|p| \leq m\\ p\in \Z}}\sum_{\substack{q \leq m\\ q\in\N}} \mathcal{F}_{\omega}[f](p,q)\, w_{p,q}(x,y) \qquad \text{and} \qquad  \mathbb{Q}_m [f](x,y):=\sum_{\substack{|p|\leq m\\p\in\Z}}\sum_{\substack{q\leq m\\q\in\N\cup\{0\}}} \mathcal{F}_{\varpi}[f](p,q)\, \varpi_{p,q}(x,y).
\end{equation}
These projections have the following properties:
\begin{lemma}\label{projectors_properties}
For $f\in L^2(\Omega)$, we have that $\mathbb{P}_m[f]$ and  $\mathbb{Q}_m[f]$ are  $C^{\infty}(\Omega)$ functions such that:
\begin{itemize}
	\item 	For $f\in H^1(\Omega)$ we have that:
\begin{align*}\label{Pm}
&\partial_x\mathbb{P}_m[f]=\mathbb{P}_m[\partial_x f],  \quad  \partial_x\mathbb{Q}_m[f]=\mathbb{Q}_m[\partial_x f],\quad \partial_y\mathbb{P}_m[f]=\mathbb{Q}_m[\partial_y f] \quad\text{and}  \quad \partial_y\mathbb{Q}_m[f]=\mathbb{P}_m[\partial_y f].
\end{align*}	
\emph{As a consequence, for $f\in H^2(\Omega)$, we have:}
	$$\partial_y^2\mathbb{P}_m[f]=\mathbb{P}_m[\partial_y^2 f] \qquad \text{and} \qquad  \partial_y^2\mathbb{Q}_m[f]=\mathbb{Q}_m[\partial_y^2 f].$$
	\item The projectors are self-adjoint in $L^2(\Omega)$:
	$$(\mathbb{P}_m[f],g)=(f,\mathbb{P}_m[g]) \quad \text{and}\quad  (\mathbb{Q}_m[f],g)=(f,\mathbb{Q}_m[g]) \quad \qquad \forall f,g\in L^{2}(\Omega).$$
\item For $f\in X^k(\Omega)$ and $g\in Y^k(\Omega)$:
\begin{align*}
||\mathbb{P}_m[f]||_{H^k(\Omega)}\leq ||f||_{H^k(\Omega)}, \quad \mathbb{P}_m[f]\to f \quad \text{in $X^k(\Omega)$}\,\\
||\mathbb{Q}_m [g]||_{H^k(\Omega)}\leq ||g||_{H^k(\Omega)},\quad \mathbb{Q}_m[f]\to f \quad \text{in $Y^k(\Omega)$}.
\end{align*}
\end{itemize}
\end{lemma}
\begin{proof} The proof is based in the arguments of the proof of Lemma \ref{Base_X^k}.\end{proof}

\section{Poisson's problem in a bounded strip}\label{Sec_3}
With all this in mind, it is time to solve  the Poisson's system with homogeneous Dirichlet condition (\ref{Poisson_Homogeneous_Neumann_psi}).
\begin{lemma}\label{Solve_Poisson}
Let $\rho\in X^k(\Omega)$. The solution of the Poisson's problem
\begin{equation*}
\left\{
\begin{array}{rllr}
\Delta \psi&= -\partial_x\bar{\rho}  \quad &\text{in}\quad &\Omega\\
\psi &=0  \hspace{1.5 cm} &\text{on} \quad &\partial\Omega \\
\end{array}
\right.
\end{equation*}
satisfies that $\psi\in X^{k+1}(\Omega)$ with $||\psi||_{H^{k+1}(\Omega)}\lesssim ||\bar{\rho}||_{H^k(\Omega)}$ and  its Fourier expansion is given by
\begin{equation}\label{expresion_psi}
\psi(x,y)=\, \sum_{p\in\Z}\sum_{q\in\N}\left( \frac{ip}{p^2+ \left(q\tfrac{\pi}{2}\right)^2 }\right)\, \mathcal{F}_{\omega}[\bar{\rho}](p,q)\, \omega_{p,q}(x,y).
\end{equation}
\end{lemma}
\begin{proof} We consider the sequence of  problems
\begin{equation*}
\left\{
\begin{array}{rllr}
\Delta \psi^{[m]}&=-\Pm[\pa_x \bar{\rho}]  \quad &\text{in}\quad &\Omega,\\
\psi^{[m]} &=0  \hspace{1.5 cm} &\text{on} \quad &\partial\Omega. \\
\end{array}
\right.
\end{equation*}
Taking $n$-derivatives with $n=0,\ldots,k$, testing again $\pa^n\psim$, integrating by parts and applying Young's inequality yields  $||\psim||_{H^{k+1}(\Omega)}\leq C||\Pm[\bar{\rho}]||_{H^k(\Omega)}\leq ||\bar{\rho}||_{H^k(\Omega)}$, since $\rho\in X^k(\Omega)$  (the constant $C$ does not depend on $m$). In addition, it is easy to check that $\pa^{n}_y\psim|_{\pa\Omega}=0$, for any even number $n$ (this is because of the definition of $\Pm$ and the boundary condition $\psim|_{\pa\Omega}=0$). These two facts allow us to pass to the limit in $m$ to find $\psi\in X^{k+1}(\Omega)$ solving \eqref{Poisson_Homogeneous_Neumann_psi}.\\

As $\bar{\rho}\in X^k(\Omega)$ and $\psi\in X^{k+1}(\Omega)$ we can expand
$$\bar{\rho}(x,y)=\sum_{p\in\Z}\sum_{q\in\N}\mathcal{F}_{\omega}[\bar{\rho}](p,q)\,\omega_{p,q}(x,y) \qquad \text{and} \qquad \psi(x,y)=\sum_{p\in\Z}\sum_{q\in\N}\mathcal{F}_{\omega}[\psi](p,q)\,\omega_{p,q}(x,y)$$
then:
\begin{align*}
-\partial_x \bar{\rho}(x,y)&=-\sum_{p\in\Z}\sum_{q\in\N} \,(ip)\,\mathcal{F}_{\omega}[\bar{\rho}](p,q)\, \omega_{p,q}(x,y),\\
\Delta \psi(x,y)&=-\sum_{p\in\Z}\sum_{q\in\N}\left(p^2+ \left(q\tfrac{\pi}{2}\right)^2 \right)\,\mathcal{F}_{\omega}[\psi](p,q)\, \omega_{p,q}(x,y).
\end{align*}
Consequently,  the following relation between the coefficients must be verified:
\begin{equation}\label{psi_rho_fourier}
\mathcal{F}[\psi](p,q)= \,\frac{ ip}{ p^2+\left( q\tfrac{\pi}{2}\right)^2 }\,\mathcal{F}_{\omega}[\bar{\rho}](p,q).
\end{equation}
\end{proof}
\begin{cor}
The velocity $\mathbf{u}=(u_1,u_2)=\nabla^\perp\psi$ from (\ref{Poisson_Homogeneous_Neumann_psi}) satisfies:
$$u_1\in Y^k(\Omega), u_2 \in X^k(\Omega)\quad  \text{and} \quad  ||\mathbf{u}||_{H^k(\Omega)}\lesssim ||\bar{\rho}||_{H^k(\Omega)}.$$
\end{cor}

\section{Local solvability of solutions in $X^k(\Omega)$}\label{Sec_4}

 To obtain a local existence result for a general smooth initial data in a general bounded domain for an \textit{active scalar} is far from being trivial. The presence of boundaries makes the well-posedness issues become more delicate (see for example \cite{Constantin-Nguyen_2} and \cite{Constantin-Nguyen_1}, in the case of SQG).\\

Here, we only focus  on our setting $\Omega$. Apart from working with the spaces $X^k(\Omega)$ and as a consequence, be careful with the special boundary conditions they impose, the proof in this section is a standard application of  Galerkin aproximations. For sake of completeness  we write the details below.\\

We return to the equations for the perturbation of the \textit{confined} IPM in $\Omega$:
\begin{equation}\label{perturbado_local_existence}
\left\{
\begin{array}{rl}
\partial_t\rho +\mathbf{u}\cdot\nabla \rho &=  u_2  \\
\mathbf{u}&=\nabla^{\perp}\psi \\
\nabla\cdot\mathbf{u}&=0
\end{array}
\right.
\end{equation}
where $\psi$ solves (\ref{Poisson_Homogeneous_Neumann_psi}) together with the no-slip condition $\mathbf{u}\cdot \mathbf{n}=0$ on $\partial\Omega$ and initial data $\rho(0)\in X^k(\Omega).$ Hence, we will prove the following result:

\begin{thm}\label{local_existence}
Let $k\in\N$ with $k\geq 3$ and an initial data $\rho(0)\in X^k(\Omega)$. Then, there exists a time $T>0$ and a constant $C$, both depending only on $||\rho||_{H^{3}(\Omega)}(0)$ and a unique solution $\rho\in C\left(0,T;X^{k}(\Omega)\right)$ of the equations (\ref{perturbado_local_existence}) such that: $$\sup_{0\leq t < T} ||\rho||_{H^{k}(\Omega)}(t)\leq C\, ||\rho||_{H^{k}(\Omega)}(0).$$
Moreover, for all $t\in[0,T)$ the following estimate holds:
\begin{equation}\label{criteria}
||\rho||_{H^{k}(\Omega)}(t)\leq ||\rho||_{H^{k}(\Omega)}(0)\,\exp\left[\widetilde{C}\int_{0}^{t}\left( ||\nabla\rho||_{L^{\infty}(\Omega)}(s)+||\nabla\mathbf{u}||_{L^{\infty}(\Omega)}(s)\right)\,ds \right].
\end{equation}
\end{thm}
\noindent
The general method of the proof is similar to that for proving existence of solutions to the Navier-Stokes and Euler equations which can be found in \cite{Majda-Bertozzi}.

The strategy of this section has two parts. First we find an approximate equation and approximate solutions that have two properties: (1) the approximate solutions exists for all time, (2) the solutions satisfy an analogous energy estimate. The second part is the passage to a limit in the approximation scheme to obtain a solution to the original equations.\\

\noindent
Before embarking on the proof, we will need some basic properties of the Sobolev spaces in bounded domains. In the next lemma, $D \subset \R^d$ is a bounded domain with smooth boundary $\partial D$.

\begin{lemma}
For $s\in\N$, the following estimates hold:
\begin{itemize}
	\item If $f,g\in H^s(D)\cap \mathcal{C}(D)$, then
\begin{equation}\label{Leibniz}
||f\,g||_{H^s(D)}\lesssim \left(||f||_{H^s(D)}\,||g||_{L^{\infty}(D)}+||f||_{L^{\infty}(D)}\,||g||_{H^s(D)}\right)
\end{equation}

	\item If $f\in H^s(D)\cap \mathcal{C}^1(D)$ and $g\in H^{s-1}(D)\cap \mathcal{C}(D)$, then for $|\alpha|\leq s$ we have that:
\begin{equation}\label{commutator}
||\partial^{\alpha} (f g)-f\partial^{\alpha} g||_{L^2(D)}\lesssim || f||_{W^{1,\infty}(D)}\, ||g||_{H^{s-1}(D)}+|| f||_{H^s(D)}\, ||g||_{L^{\infty}(D)}
\end{equation}
\end{itemize}
Moreover, the following Sobolev embeddings hold:
\begin{itemize}
	\item $W^{s,p}(D)\subseteq L^{q}(D)$ continuously if $s<n/p$ and $p\leq q \leq np/(n-sp)$.
	
	\item  $W^{s,p}(D)\subseteq \mathcal{C}^{k}(\overline{D})$ consinuously is $s>k+n/p.$
\end{itemize}
\end{lemma}
\begin{proof}
 See \cite[p.~280]{Ferrari} and references therein.
\end{proof}

\begin{proof}[Proof of Theorem \ref{local_existence}]

We firstly construct  approximate
equations by using a smoothing procedure called Galerkin method.
The $m^{\text{th}}$-Galerkin approximation of (\ref{perturbado_local_existence}) is the following  system:

\begin{equation}\label{rho_sharp}
\left\{
\begin{array}{rl}
\partial_t\rho^{[m]} +\mathbb{P}_m\left[\mathbf{u}^{[m]}\cdot\nabla \rho^{[m]} \right]&=  u_2^{[m]}  \\
\mathbf{u}^{[m]}&=\nabla^\perp \psi^{[m]}\\
\rho^{[m]}|_{t=0}&=\mathbb{P}_m[\rho](0)
\end{array}
\right.
\end{equation}
where
\begin{equation}\label{Poisson_rho_sharp}
\left\{
\begin{array}{rllr}
\Delta \psi^{[m]}&= -\partial_x\rho^{[m]}  \quad &\text{in}\quad &\Omega\,\\
\psi^{[m]} &=0  \hspace{1.5 cm} &\text{on} \quad &\partial\Omega. \\
\end{array}
\right.
\end{equation}
and with $\rho(0)\in X^k(\Omega)$. Since the initial data in \eqref{rho_sharp} belongs to $\Pm L^2(\Omega)$ and because of the structure of the equations, we look for solutions of the form: $$\rho^{[m]}(t)=\sum_{\substack{|p| \leq m\\ p\in \Z}}\sum_{\substack{q \leq m\\ q\in\N}} \mathfrak{c}^{[m]}_{p,q}(t) \omega_{p,q}(x,y).$$
Then, by Lemma (\ref{Solve_Poisson}) we get:
$$\psi^{[m]}(t)=\sum_{\substack{|p| \leq m\\ p\in \Z}}\sum_{\substack{q \leq m\\ q\in\N}}\left(\frac{ip}{p^2+\left(q\tfrac{\pi}{2}\right)^2}  \right) \mathfrak{c}^{[m]}_{p,q}(t) \omega_{p,q}(x,y).$$

In this way, \eqref{rho_sharp} is reduced to a finite dimensional ODE system for the coefficients $\mathfrak{c}^{[m]}_{p,q}(t)$ for $\{|p|, q\}\leq m$, and we can apply Picard's theorem to find a solution on a time of existence depending on $m$. Next, we will use energy estimates to prove that there is a time of existence $T$,  uniform in $m$, for every solution $\rho^{[m]}(t)$ of \eqref{rho_sharp} and a limit $\rho(t)$ which will solve $\eqref{perturbado_local_existence}$. To do it, we recall that:
$$\rho^{[m]}=\mathbb{P}_m\left[\rho^{[m]}\right] \qquad \text{and} \qquad \mathbf{u}^{[m]}=\left(u_1^{[m]},u_2^{[m]}\right)=\left(\mathbb{Q}_m\left[u_1^{[m]}\right],\mathbb{P}_m\left[u_2^{[m]}\right]\right).$$
Taking derivatives $\partial^{s}$, with $|s|\leq k$ on the first equation of (\ref{rho_sharp}) and then taking the  $L^2(\Omega)$ inner product with $\partial^s \rho^{[m]}$, we obtain:
\begin{equation*}
\left( \partial_t \partial^s \rho^{[m]},\partial^s \rho^{[m]}\right)=\left( \partial^s u_2^{[m]},\partial^s \rho^{[m]}\right) - \left(\partial^s \mathbb{P}_m  \left[\mathbf{u}^{[m]}\cdot\nabla \rho^{[m]} \right],\partial^s \rho^{[m]}\right)=I-II.
\end{equation*}
For the first term, since $\psi^{[m]}$ solves the Poisson's problem (\ref{Poisson_rho_sharp}), integrations by parts give us:
\begin{equation}\label{I}
I=\left( \partial^s \partial_x\psi^{[m]}, \partial^s\rho^{[m]}\right)=\left(\partial^s \psi^{[m]},\partial^s \Delta \psi^{[m]}\right)=-||\pa^s\mathbf{u}^{[m]}||_{L^2(\Omega)}^2
\end{equation}
thanks to $\pa^{n}_y\psim|_{\pa\Omega}=0$, for any even number $n$.
For the second one, we need to distinguish between an even or odd number of $y$-derivatives. In any case, the properties of $\mathbb{P}_m, \mathbb{Q}_m$ given by Lemma (\ref{projectors_properties}) and the commutator estimate (\ref{commutator}) with $f=\mathbf{u}^{[m]}$ and $g=\nabla\rho^{[m]}$ give us the inequality:
\begin{equation}\label{II}
II\lesssim ||\pa^s\rho^{[m]}||_{L^2(\Omega)}\left( ||\nabla \mathbf{u}^{[m]}||_{L^{\infty}(\Omega)}||\rho^{[m]}||_{H^k(\Omega)}+||\mathbf{u}^{[m]}||_{H^k(\Omega)}||\nabla \rho^{[m]}||_{L^{\infty}(\Omega)} \right).
\end{equation}
Summing over $|s|\leq k$ and putting together (\ref{I}) and (\ref{II}) we obtain:
\begin{align*}\label{estimate_m_BKM}
\tfrac{1}{2}\partial_t ||\rho^{[m]}||_{H^{k}(\Omega)}^2 &\lesssim ||\rho^{[m]}||_{H^{k}(\Omega)}\left( ||\nabla \mathbf{u}^{[m]}||_{L^{\infty}(\Omega)}||\rho^{[m]}||_{H^k(\Omega)}+||\mathbf{u}^{[m]}||_{H^k(\Omega)}||\nabla \rho^{[m]}||_{L^{\infty}(\Omega)} \right)\nonumber
\end{align*}
and as $\mathbf{u}^{[m]}=\nabla^{\perp}\psi^{[m]}$ where $\psi^{[m]}$ solves (\ref{Poisson_rho_sharp}) by Lemma (\ref{Solve_Poisson}) we get the bound $||\mathbf{u}^{[m]}||_{H^k(\Omega)} \lesssim ||\rho^{[m]}||_{H^k(\Omega)}$.
Therefore, we finally obtain that:
\begin{equation}\label{estimate_Hk}
\tfrac{1}{2}\partial_t ||\rho^{[m]}||_{H^{k}(\Omega)}^2 \lesssim ||\rho^{[m]}||_{H^{k}(\Omega)}^2\left( ||\nabla \mathbf{u}^{[m]}||_{L^{\infty}(\Omega)}+||\nabla \rho^{[m]}||_{L^{\infty}(\Omega)} \right)\lesssim ||\rho^{[m]}||_{H^{k}(\Omega)}^2 ||\rho^{[m]}||_{H^3(\Omega)}
\end{equation}
where the last inequality is true provided that $k\geq 3$ due to the Sobolev embedding $L^{\infty}(\Omega)\hookrightarrow H^{2}(\Omega)$.
Hence, for all $m$ and $0\leq t < T\leq\left( c\, ||\rho||_{H^{3}(\Omega)}(0)\right)^{-1}$ we have that:
\begin{equation}\label{estimate_H3}
||\rho^{[m]}||_{H^{3}(\Omega)}(t)\leq \frac{||\mathbb{P}_m[\rho]||_{H^{3}(\Omega)}(0)}{1-c\,t\, ||\mathbb{P}_m[\rho]||_{H^{3}(\Omega)}(0)}\leq \frac{||\rho||_{H^{3}(\Omega)}(0)}{1-c\,t\, ||\rho||_{H^{3}(\Omega)}(0)}
\end{equation}
and, in particular	
$$\sup_{0\leq t< T}||\rho^{[m]}||_{H^{3}(\Omega)}(t)\leq \frac{||\rho||_{H^{3}(\Omega)}(0)}{1-c\,T\, ||\rho||_{H^{3}(\Omega)}(0)}.$$
Applying (\ref{estimate_H3}) in the last term of (\ref{estimate_Hk}), we obtain for all $m$ and $0\leq t< T$ by Gronwall's lemma that:
\begin{align*}
||\rho^{[m]}||_{H^{k}(\Omega)}(t)&\leq ||\mathbb{P}_m[\rho^{[m]}]||_{H^{k}(\Omega)}(0)\,\exp\left[c \int_{0}^t \frac{||\rho||_{H^{3}(\Omega)}(0)}{1-c\,s\, ||\rho||_{H^{3}(\Omega)}(0)}\, ds \right]\nonumber\\
&\leq ||\rho||_{H^{k}(\Omega)}(0)\,\exp\left[c \int_{0}^t \frac{||\rho||_{H^{3}(\Omega)}(0)}{1-c\,s\, ||\rho||_{H^{3}(\Omega)}(0)}\, ds \right]
\end{align*}
and, in particular	
\begin{equation}\label{control_H^k_time_t0}
\sup_{0\leq t< T}||\rho^{[m]}||_{H^{k}(\Omega)}(t)\leq C\,||\rho||_{H^{k}(\Omega)}(0)
\end{equation}
where $C$ is a constant depending only on $||\rho||_{H^3(\Omega)}(0).$\\

Therefore, the family $\rho^{[m]}$ is uniformly bounded, with respect to $m$, in $L^{\infty}\left(0,T;H^{k}(\Omega)\right)$. One consequence of the Banach-Alaoglu theorem  (see \cite{Royden}) is that a bounded sequence $||\rho^{[m]}||_{H^{k}(\Omega)}\leq K$ has a subsequence that converges weakly to some limit in $H^{k}(\Omega)$, which is the dual of a separable Ba\nolinebreak nach space. This is $\rho^{[m]}(t)\rightharpoonup\rho(t)$ in $H^{k}(\Omega)$ for $0\leq t< T$.\\

Moreover, the family  $\partial_t \rho^{[m]}$ is uniformly bounded in $L^{\infty}\left(0,T;H^{k-2}(\Omega)\right)$. By (\ref{rho_sharp}) we have that:
\begin{align*}
\sup_{0\leq t< T}||\partial_t \rho^{[m]}||_{H^{k-2}(\Omega)}(t)&=\sup_{0\leq t< T}||u_2^{[m]}-\mathbb{P}_m\left[\mathbf{u}^{[m]}\cdot\nabla \rho^{[m]}\right]||_{H^{k-2}(\Omega)}(t) \\
&\leq \sup_{0\leq t< T}|| u_2^{[m]}||_{H^{k-2}(\Omega)}(t)+\sup_{0\leq t< T}||\mathbb{P}_m\left[\mathbf{u}^{[m]}\cdot\nabla \rho^{[m]} \right]||_{H^{k-2}(\Omega)}(t).
\end{align*}
 We need to show that $\mathbf{u}^{[m]}\cdot\nabla \rho^{[m]}\in X^{k-1}(\Omega)$ to apply Lemma (\ref{projectors_properties}), for $k\geq 3$, and to get:
\begin{align*}
||\mathbb{P}_m\left[\mathbf{u}^{[m]}\cdot\nabla \rho^{[m]} \right]||_{H^{k-2}(\Omega)}(t)&\leq ||\mathbf{u}^{[m]}\cdot\nabla \rho^{[m]} ||_{H^{k-2}(\Omega)}(t)\\
&\lesssim  \left[||\mathbf{u}^{[m]}||_{H^{k-2}(\Omega)}\,||\nabla \rho^{[m]} ||_{L^{\infty}(\Omega)}+||\mathbf{u}^{[m]}||_{L^{\infty}(\Omega)}\,||\nabla \rho^{[m]}||_{H^{k-2}(\Omega)}\right](t)\\
&\lesssim ||\mathbf{u}^{[m]}||_{H^{k}(\Omega)}(t)\, ||\rho^{[m]} ||_{H^{k}(\Omega)}(t)
\end{align*}
where in the last inequalities we have used (\ref{Leibniz}) and the Sobolev embedding $L^{\infty}(\Omega)\hookrightarrow H^{2}(\Omega)$.

  Checking that $\mathbf{u}^{[m]}\cdot\nabla \rho^{[m]}\in X^{k-1}(\Omega)$  reduces to see that $\partial_y^{n}\left(\mathbf{u}^{[m]}\cdot\nabla \rho^{[m]}\right)|_{\partial\Omega}=0$ for any even natural number $n$. We start, with the following observation:
$$\mathbf{u}^{[m]}\cdot\nabla \rho^{[m]}=\mathbb{Q}_m \left[u_1^{[m]}\right]\mathbb{P}_m\left[\pa_x\rho^{[m]}\right]+\mathbb{P}_m\left[u_2^{[m]}\right] \mathbb{Q}_m\left[\pa_y \rho^{[m]}\right]$$
and the fact, due to (\ref{pa_b}) and (\ref{pa_c}), that:
\begin{align*}
\pa_y^2(b_q \,c_q)(y)&=(\pa_y^2 b_q)(y)\, c_q(y)+2 (\pa_y b_q)(y)\, (\pa_y c_q)(y)+b_q(y)\, (\pa_y^2 c_q)(y)\\
&=(-1)(q\,\pi)^2 b_q(y)\,c_q(y).
\end{align*}
Iterating this procedure and using that $b_q(\pm 1)=0$ we prove the boundary conditions for the derivatives of even order of the non-linear term.

As before, by Lemma (\ref{Solve_Poisson}) we obtain the bound $||\mathbf{u}^{[m]}||_{H^{k}(\Omega)}(t)\lesssim ||\rho^{[m]}||_{H^{k}(\Omega)}(t)$ and putting all together we obtain:
\begin{align*}
\sup_{0\leq t< T}||\partial_t \rho^{[m]}||_{H^{k-2}(\Omega)}(t)&\lesssim \sup_{0\leq t< T} ||\rho^{[m]}||_{H^{k}(\Omega)}(t)\left[1+||\rho^{[m]}||_{H^{k}(\Omega)}(t)\right]\\
&\leq C \,||\rho||_{H^{k}(\Omega)}(0) \left[1+C \,||\rho||_{H^{k}(\Omega)}(0)\right]
\end{align*}
thanks to (\ref{control_H^k_time_t0}). Hence, the family of time derivatives $\partial_t \rho^{[m]}(t)$ is uniformly bounded in $L^{\infty}\left(0,T;H^{k-2}(\Omega)\right)$.

Therefore, as we have seen above, the family of time derivatives $\partial_t \rho^{[m]}(t)$ is uniformly bounded in $L^{\infty}\left(0,T;H^{k-2}(\Omega)\right)$. Then, by Banach-Alaoglu theorem, $\pa_t\rho^{[m]}(t)$ has a subsequence that converges weakly to some limit in $H^{k-2}(\Omega)$ for $0\leq t<T$.

Moreover, by virtue of Aubin-Lions's compactness lemma  (see for instance \cite{Lions}) applied with the triple $H^{k}(\Omega)\subset\subset H^{k-1}(\Omega)\subset H^{k-2}(\Omega)$ we obtain that the convergence $\rho^{[m]} \rightarrow \rho$ is strong in $C(0,T;H^{k-1}(\Omega))$.
As $\mathbf{u}^{[m]}=\nabla^{\perp}\psi^{[m]}$ where $\psi^{[m]}$ solves (\ref{Poisson_rho_sharp}) and the convergence $\rho^{[m]} \rightarrow \rho$ is strong in $C(0,T;H^{k-1}(\Omega))$, we obtain the strong convergence $\mathbf{u}^{[m]} \rightarrow \mathbf{u}$ in $C(0,T; Y^{k-1}(\Omega)\times X^{k-1}(\Omega))$. Using these facts, we may pass to the limit in the non-linear part of (\ref{rho_sharp}) to see that $\mathbb{P}_m[\mathbf{u}^{[m]}\cdot\nabla\rho^{[m]}]\rightarrow \mathbf{u}\cdot\nabla\rho$ in $C(0,T; H^{k-2}(\Omega))$ as follows:
\begin{align*}
||\mathbb{P}_m[\mathbf{u}^{[m]}\cdot\nabla\rho^{[m]}]- \mathbf{u}\cdot\nabla\rho||_{H^{k-2}(\Omega)}&=||\mathbb{P}_m[\mathbf{u}^{[m]}\cdot\nabla\rho^{[m]}] \pm \mathbf{u}^{[m]}\cdot\nabla\rho^{[m]}\pm \mathbf{u}^{[m]}\cdot\nabla\rho - \mathbf{u}\cdot\nabla\rho||_{H^{k-2}(\Omega)}\\
&\leq \big|\big|(\mathbb{P}_m-\mathbb{I})[\mathbf{u}^{[m]}\cdot\nabla\rho^{[m]}]\big|\big|_{H^{k-2}(\Omega)}+\big|\big| \mathbf{u}^{[m]}\cdot\nabla(\rho^{[m]}-\rho)\big|\big|_{H^{k-2}(\Omega)}\\
&+\big|\big|(\mathbf{u}^{[m]}-\mathbf{u})\cdot\nabla\rho\big|\big|_{H^{k-2}(\Omega)}\rightarrow 0 \qquad \text{as}\quad  m\to \infty.
\end{align*}
In the limit, we use the fact that $\lim_{m\to\infty} ||\mathbb{P}_m[f]-f||_{H^s(\Omega)}=0$ for $f\in X^s(\Omega)$, together with the convergences of $\mathbf{u}^{[m]}\rightarrow \mathbf{u}$ and $\rho^{[m]}\rightarrow \rho$ and \eqref{Leibniz}, for  $k\geq 3$.\\

Now, from (\ref{rho_sharp}), we have that
$\partial_t\rho^{[m]} =  u_2^{[m]}-\mathbb{P}_m\left[\mathbf{u}^{[m]}\cdot\nabla \rho^{[m]} \right]\rightarrow u_2- \mathbf{u}\cdot\nabla\rho$ in $C(0,T;H^{k-2}(\Omega))$.
Since $\rho^{[m]}\rightarrow \rho$ in $C(0,T;H^{k-1}(\Omega))$, the distribution limit of $\pa_t \rho^{[m]}$ must be  $\pa_t\rho$  for the Closed Graph theorem \cite{Brezis}. So, in particular, it follows that $\rho(t)$ is the unique classical solution of (\ref{perturbado_local_existence}) which lies in $C(0,T; H^{k-1}(\Omega))$.
Then, to show that $\rho\in C(0,T; H^{k}(\Omega))$ we follow \cite[p.~110]{Majda-Bertozzi}.

Firstly, we recall that $\rho\in L^{\infty}(0,T;H^k(\Omega))\subset L^{2}(0,T;H^k(\Omega))$ and we start proving that
$\rho(t)$ is continuous on $[0,T)$ in the weak topology of $H^k(\Omega)$. To prove that  $\rho\in C_{W}(0,T; H^{k}(\Omega))$, we define the dual pairing  of $(H^{s})^{\star}(\Omega)$ and $H^{s}(\Omega)$ as $[\cdot,\cdot]: (H^{s}(\Omega))^{\star}\times H^{s}(\Omega)\rightarrow \R$ given by $[\varphi,f]:=\varphi[f]$. Hence, because $\rho^{[m]}\rightarrow \rho $ in $C(0,T;H^{k-1}(\Omega))$, it follows that $[\varphi,\rho^{[m]}(t)]\rightarrow [\varphi,\rho(t)]$ uniformly on $[0,T)$ for any $\varphi \in (H^{k-1}(\Omega))^{\star}$.

Using the fact that $(H^{k-1}(\Omega))^{\star}$ is dense in $(H^{k}(\Omega))^{\star}$ by means of an $\epsilon$-argument together with (\ref{control_H^k_time_t0}), we have
$[\varphi,\rho^{[m]}]\rightarrow [\varphi,\rho]$ uniformly on $[0,T)$ for any $\varphi\in (H^{k}(\Omega))^{\star}$. This fact implies that $\rho \in C_{W}(0,T;H^k(\Omega))$.

By virtue of the fact that $\rho\in C_{W}(0,T; H^{k}(\Omega))$, it suffices to show that the norm $||\rho||_{H^k(\Omega)}(t)$ is a contin\nolinebreak uous function of time to get that $\rho\in C(0,T; H^{k}(\Omega))$.\\

\noindent
Recall the relation for the uniform $H^{k}(\Omega)$ norm  for the approximations:
$$||\rho^{[m]}||_{H^k(\Omega)}(t)\leq \frac{||\rho||_{H^k(\Omega)}(0)}{1-C\, t\, ||\rho||_{H^k(\Omega)}(0)}=||\rho||_{H^k(\Omega)}(0)+\frac{C\, t\, ||\rho||_{H^k(\Omega)}^2(0)}{1-C\, t\, ||\rho||_{H^k(\Omega)}(0)} \qquad \text{for all }\quad 0\leq t< T.$$
For fixed time $t\in [0,T)$ we have $||\rho||_{H^k(\Omega)}(t)\leq \liminf_{m\to\infty}||\rho^{[m]}||_{H^k(\Omega)}(t)$. Using it in the above expression, we obtain:
$$||\rho||_{H^{k}(\Omega)}(t)\leq ||\rho||_{H^k(\Omega)}(0)+\frac{C\, t\, ||\rho||_{H^k(\Omega)}^2(0)}{1-C\, t\, ||\rho||_{H^k(\Omega)}(0)}.$$
On one hand, from the fact that $\rho\in C_{W}(0,T;H^k(\Omega))$, we have that $||\rho||_{H^k(\Omega)}(0)\leq \liminf_{t\to 0^{+}} ||\rho||_{H^k(\Omega)}(t)$. On the other hand, the above expression gives us that $\limsup_{t\to 0^{+}} ||\rho||_{H^k(\Omega)}(t)\leq ||\rho||_{H^k(\Omega)}(0)$. Then, in particular $\lim_{t\to 0^{+}} ||\rho||_{H^k(\Omega)}(t)=||\rho||_{H^k(\Omega)}(0)$. This gives us strong right continuity at $t=0$.

It remains to prove continuity of the $||\cdot||_{H^k(\Omega)}(t)$ norm of the solution at times other than
the initial time. Consider a time $t^{\star}\in (0,T)$ and the solution $\rho(t^{\star})\in H^k(\Omega)$. At this fixed time, we define $\rho^{\star}(0):=\rho(t^{\star})$. So we can take, $\rho^{\star}(0)$ as initial data and construct solution as above by solving regularized equation (\ref{rho_sharp}). Following the argument we used above
 to show that $||\rho||_{H^k(\Omega)}(t)$ is continuous at $t=0$, we also conclude that it is continuous as $t=t^{\star}$.
Because $t^{\star}\in(0,T)$ is arbitrary, we have just showed that $||\rho||_{H^k(\Omega)}(t)$ is a continuous function on $[0,T)$. In consequence, we have proved that $\rho\in C(0,T;H^k(\Omega))$.

Since for every $m\in\N$ we have $\rho^{[m]}=\mathbb{P}_m[\rho^{[m]}]\in X^{k}(\Omega)$, i.e. $\partial_y^{n}\rho^{[m]}|_{\partial\Omega}=0$ for any even number $n$ and this property is closed, we obtain that the limiting function also has the desired property, which concludes that the solution $\rho$ lies in $C\left(0,T;X^{k}(\Omega)\right).$\\

Finally, applying the Gronwall's lemma on the above estimate (\ref{estimate_Hk}) and the previous convergence results, for all $t\in[0,T)$ we deduce:
\begin{align*}
||\rho^{[m]}||_{H^{k}(\Omega)}(t)&\leq ||\rho^{[m]}||_{H^{k}(\Omega)}(0)\,\exp\left[\widetilde{C} \int_{0}^{t}\left( ||\nabla\rho^{[m]}||_{L^{\infty}(\Omega)}(s)+||\nabla\mathbf{u}^{[m]}||_{L^{\infty}(\Omega)}(s)\right)\,ds \right]\\
&\leq ||\rho||_{H^{k}(\Omega)}(0)\,\exp\left[\widetilde{C}\int_{0}^{t}\left( ||\nabla\rho||_{L^{\infty}(\Omega)}(s)+||\nabla\mathbf{u}||_{L^{\infty}(\Omega)}(s)\right)\,ds \right]
\end{align*}
and by lower
semicontinuity we obtain (\ref{criteria}).

\end{proof}

\begin{thm}
If $\rho(t)$ is a solution of (\ref{perturbado_local_existence}) in the class $C\left(0,T,X^{k}(\Omega)\right)$ with $\rho(0)\in X^k(\Omega)$, and if $T=T^{\star}$ is the first time such that $\rho(t)$ is not contained in this class, then
\begin{equation*}
\int_{0}^{T^{\star}}\left(||\nabla\mathbf{u}||_{L^{\infty}(\Omega)}(s)+||\nabla\rho||_{L^{\infty}(\Omega)}(s)\right)\, ds=\infty.
\end{equation*}
\end{thm}

\begin{proof}
This result follows from estimate \eqref{criteria}.
\end{proof}

\section{Global regularity for small initial data}\label{Sec_5}
This section is devoted to prove the main result of this paper:

\begin{thm}\label{main_thm}
Let $\Theta(y):=-y$. There exists $\varepsilon_0>0$ and a parameter $\gamma\in\N$ with $\gamma>4$ such that if we solve (\ref{System_A}) with initial data $\varrho(0)=\Theta+\rho(0)$ and $\rho(0)\in X^{\kappa}(\Omega)$ with $||\rho||_{H^{\kappa}(\Omega)}(0)< \varepsilon\leq \varepsilon_0$ where $\kappa\geq 5+\gamma$ then, the solution exists globally in time and satisfies the following:
\begin{enumerate}
	\item $||\bar{\varrho}||_{H^{3}(\Omega)}(t)\equiv ||\bar{\rho}||_{H^{3}(\Omega)}(t)\lesssim \frac{\varepsilon}{(1+t)^{\frac{\gamma}{4}}}$
	
	\item $||\tilde{\varrho}-\Theta||_{H^{\kappa}(\Omega)}(t)\equiv ||\tilde{\rho}||_{H^{\kappa}(\Omega)}(t)\leq 2\,\varepsilon$
\end{enumerate}
where $\varrho:=\bar{\varrho}+\tilde{\varrho}$ such that $\bar{\varrho}\perp\tilde{\varrho}$ and $\bar{\varrho}$ is given by the projection operator onto the subspace of functions with zero average in the horizontal variable.
\end{thm}

In the next three sections we give the proof of this result.

\subsection{Energy methods for the confined IPM equation}\label{Sec_5.1}

From what we have seen, we know that for $\rho(0)\in X^k(\Omega)$ there exits $T>0$ such that $\rho(t)\in X^k(\Omega)$ is a solution of (\ref{System_A_PI}) for all $t\in[0,T)$. Moreover, if $T^{\star}$ is the first time such that $\rho(t)$ is not contained in this class, then
$$\int_{0}^{T^{\star}}\left(||\nabla\mathbf{u}||_{L^{\infty}(\Omega)}(s)+||\nabla\rho||_{L^{\infty}(\Omega)}(s)\right)\, ds=\infty.$$

\subsubsection{A priori energy estimate}In  what follows,  we  assume that $\rho(t)\in X^k(\Omega)$ is a solution of (\ref{System_A_PI}) for any $t\geq 0$. Then, the following estimate holds for $k\geq 6$:
$$\tfrac{1}{2}\partial_t ||\rho||_{H^{k}(\Omega)}^2(t) \lesssim  ||\partial u_2||_{L^{\infty}(\Omega)}(t)\, ||\rho||_{H^{k}(\Omega)}^2(t)- \left(1-||\rho||_{H^{k}(\Omega)}(t)\right)\, ||\mathbf{u}||_{H^{k}(\Omega)}^2(t).$$

\noindent
In this section we will perform the basic energy estimate for
\begin{equation}\label{Evolucion_rho}
\partial_t\rho +\mathbf{u}\cdot\nabla \rho = u_2.
\end{equation}
\textbf{$L^2(\Omega)$-estimate:} We begin with the $L^2(\Omega)$ bound. We multiply (\ref{Evolucion_rho}) by $\rho$ and integrate over $\Omega$. Then,
$$\tfrac{1}{2}\partial_t ||\rho||_{L^2(\Omega)}^2=\int_{\Omega}\rho \, \partial_t\rho\,dxdy=\int_{\Omega} \rho\, u_2\,dxdy-\int_{\Omega}\rho\, \left(\mathbf{u}\cdot\nabla\right)\rho\,dxdy.$$
By the incompressibility of the velocity and the boundary conditions, we have that the second term vanish.
So, by (\ref{Stream_1}) we get:
$$\tfrac{1}{2}\partial_t ||\rho||_{L^2(\Omega)}^2= \int_{\Omega} \rho\, u_2\,dxdy= \int_{\Omega} \rho\, \partial_x\psi\,dxdy.$$
Finally, applying integration by parts and using that $\psi$ solves (\ref{Poisson_Homogeneous_Neumann_psi}) we achieve:
$$\tfrac{1}{2}\partial_t ||\rho||_{L^2(\Omega)}^2= \int_{\Omega} \Delta\psi \, \psi\,dxdy=-\int_{\Omega}\left(\nabla \psi \right)^2\, dx dy+ \int_{\Omega}\partial_y[\partial_y\psi \, \psi]\, dxdy.$$
As $\psi|_{\partial\Omega}=0$, it is clear that the boundary term vanish, and consequently, we have that:
\begin{equation}
\tfrac{1}{2}\partial_t ||\rho||_{L^2(\Omega)}^2=- \,||\nabla \psi ||_{L^2(\Omega)}^2.
\end{equation}

\noindent
\textbf{$\dot{H}^{k}(\Omega)$-estimate:} We next take $\partial^{k}$ to (\ref{Evolucion_rho}), we multiply it by $\partial^{k}\rho$ and integrate over $\Omega$. Then,
\begin{align*}
\tfrac{1}{2}\partial_t ||\rho||_{\dot{H}^{k}(\Omega)}^2=\int_{\Omega}\partial^{k}\rho\,\partial_t\partial^{k}\rho\,dxdy&= \int_{\Omega}\partial^{k}\rho\,\partial^{k} u_2dxdy-\int_{\Omega}\partial^{k}\rho\,\partial^{k}(\mathbf{u}\cdot\nabla)\rho\, dxdy\\
&=I_1+I_2.
\end{align*}
First of all, we study $I_1$. By (\ref{Stream_1}), (\ref{Poisson_Homogeneous_Neumann_psi}) and integration by parts, we get:
\begin{align*}
I_1&=\int_{\Omega}\partial^{k}\rho\,\partial^{k}\partial_x \psi dxdy=-\int_{\Omega}\partial^{k}\partial_x\rho\,\partial^{k} \psi dxdy=\int_{\Omega}\Delta\partial^{k}\psi\,\partial^{k} \psi dxdy\\
&=-\int_{\Omega}\left(\nabla\partial^{k} \psi\right)^2 dxdy+\int_{\Omega}\partial_y\left[\partial_y\partial^{k}\psi \, \partial^{k}\psi\right] \, dxdy.
\end{align*}
As $\psi\in X^{k+1}(\Omega)$ due to Lemma (\ref{Solve_Poisson}), the boundary term vanish and we have proved that:
\begin{equation}
I_1=- \,||\nabla \psi ||_{\dot{H}^{k}(\Omega)}^2.
\end{equation}

\noindent
Secondly, we study $I_2$. The most singular term vanish by the incompressibility and the boundary conditions:
\begin{align*}
I_2&=-\int_{\Omega}\partial^{k}\rho\,\partial^{k}(\mathbf{u}\cdot\nabla)\rho\, dxdy\\
&=-\int_{\Omega}\partial^{k} \rho \, \left(\partial\mathbf{u}\cdot\nabla\partial^{k-1}\rho\right) \, dxdy -\sum_{i=1}^{k-1}{k \choose i}\int_{\Omega} \partial^{k}\rho \, \left(\partial^{i+1}\mathbf{u}\cdot\nabla\partial^{k-i-1}\rho\right) \, dx dy.
\end{align*}
Now, we want to distinguish between two kinds of terms, first the case where $i=0$ and then the case where $1\leq i\leq k-1$. The term for $i=0$ is bounded directly as:
\begin{equation*}
-\int_{\Omega}\partial^{k} \rho \, \left(\partial\mathbf{u}\cdot\nabla\partial^{k-1}\rho\right) \, dxdy\leq ||\partial \mathbf{u}||_{L^{\infty}(\Omega)}\, ||\rho||_{H^{k}(\Omega)}^2
\end{equation*}
but working a little bit harder, we achieve:
\begin{align*}
-\int_{\Omega}\partial^{k} \rho \, \left(\partial\mathbf{u}\cdot\nabla\partial^{k-1}\rho\right) \, dxdy&=-\int_{\Omega}\partial^{k} \rho \, \left(\partial u_1\,\partial_x\partial^{k-1}\rho+\partial u_2\,\partial_y\partial^{k-1}\rho\right) \, dxdy\\
&\leq \int_{\Omega}\partial^{k} \rho \, \left(\partial u_1\,\partial_x\partial^{k-1}\rho\right) \, dxdy+||\partial u_2||_{L^{\infty}(\Omega)}\, ||\rho||_{H^{k}(\Omega)}^2
\end{align*}
where, for the first integral, we consider two cases:
\begin{itemize}
	\item \Ovalbox{$\partial u_1 \equiv \partial_x u_1$} By the incompressibility of the flow it is clear that:
	$$\int_{\Omega}\partial^{k} \rho \, \left(\partial_x u_1\,\partial_x\partial^{k-1}\rho\right) \, dxdy=-\int_{\Omega}\partial^{k} \rho \, \left(\partial_y u_2\,\partial_x\partial^{k-1}\rho\right) \, dxdy\leq ||\partial u_2||_{L^{\infty}(\Omega)}\, ||\rho||_{H^{k}(\Omega)}^2.$$
	\item \Ovalbox{$\partial u_1 \equiv \partial_y u_1$} In this case, by (\ref{Stream_1}) we have that:
\begin{align*}
\int_{\Omega}\partial^{k} \rho \, \left(\partial_y u_1\,\partial_x\partial^{k-1}\rho\right) \, dxdy&=\int_{\Omega}\partial^{k} \rho \, \left(\partial_x u_2\,\partial_x\partial^{k-1}\rho\right) \, dxdy \, +\int_{\Omega}\partial^{k} \rho \, \left(\partial_x \rho\,\partial_x\partial^{k-1}\rho\right) \, dxdy\\
	&\leq ||\partial u_2||_{L^{\infty}(\Omega)}\, ||\rho||_{H^{k}(\Omega)}^2+||\rho||_{H^{k}(\Omega)}\,||\nabla\psi||_{H^{k}}^2.
\end{align*}
\end{itemize}
To sum up, we have proved that:
\begin{equation}\label{SPECIAL_TERM}
-\int_{\Omega}\partial^{k} \rho \, \left(\partial\mathbf{u}\cdot\nabla\partial^{k-1}\rho\right) \, dxdy\leq ||\partial u_2||_{L^{\infty}(\Omega)}\, ||\rho||_{H^{k}(\Omega)}^2+||\rho||_{H^{k}(\Omega)}\,||\nabla\psi||_{H^{k}}^2.
\end{equation}
\noindent
Indeed, this is the only term that cannot be absorbed by the linear part. This term is the reason why we need to have a integrable time decay  of $||\partial u_2||_{L^{\infty}(\Omega)}$. Precisely, the main goal of the next Section \ref{Sec_5.2} is to obtain a time decay rate for it. \\

By the other hand, for $i=1,\ldots,k-1$ we separate the other term as follows:
\begin{align*}
\int_{\Omega} \partial^{k}\rho \, \left(\partial^{i+1}\mathbf{u}\cdot\nabla\partial^{k-i-1}\rho\right) \, dx dy&=\int_{\Omega} \partial^{k}\rho \, \partial^{i+1}u_1\,\partial_x\partial^{k-i-1}\rho \, dx dy+\int_{\Omega} \partial^{k}\rho \, \partial^{i+1}u_2\,\partial_y\partial^{k-i-1}\rho \, dx dy\\
&=J_1(i)+J_2(i) \qquad i=1,\ldots,k-1.
\end{align*}

\noindent
In view of (\ref{Stream_1}) and (\ref{Poisson_Homogeneous_Neumann_psi}), we have that $J_1(i)$ can be rewritten as:
\begin{align*}
J_1(i)&=\int_{\Omega} \partial^{k}\rho \, \partial^{i+1}\partial_y\psi \,\partial^{k-i-1}\Delta\psi \, dx dy
\end{align*}
and we clearly have
\begin{align}\label{J1}
\sum_{i=1}^{k-1}J_1(i)&\leq  ||\partial^{k}\rho||_{L^2}\left[ \sum_{i=1}^{k-3} ||\partial^{i+1}\partial_y\psi||_{L^{\infty}} ||\partial^{k-i-1}\Delta\psi||_{L^2}+\sum_{i=k-2}^{k-1} ||\partial^{i+1}\partial_y\psi||_{L^{2}} ||\partial^{k-i-1}\Delta\psi||_{L^{\infty}} \right] \nonumber \\
&\leq ||\rho||_{H^{k}(\Omega)}\, ||\nabla\psi||_{H^{k}(\Omega)}^2 \qquad \text{for}\qquad k\geq 4.
\end{align}

\noindent
For $J_2(i)$, by (\ref{Stream_1}) we obtain that:
\begin{align*}
J_2(i)&=\int_{\Omega} \partial^{k}\rho \, \partial^{i+1}\partial_x\psi \,\partial_y\partial^{k-i-1}\rho \, dx dy
\end{align*}

\noindent
and for $i=1,\ldots, k-1$ we need to distinguish two situations:\\

$\hspace{-1.1 cm}\bullet$ \quad  We have at least one derivative in $x$. This is $\partial^{k}\equiv \partial^{k-1}\partial_x$. Then, by (\ref{Poisson_Homogeneous_Neumann_psi}) we can write $J_2(i)$ as follows:
\begin{align*}
J_2(i)&=-\int_{\Omega} \partial^{k-1}\Delta\psi \, \partial^{i+1}\partial_x\psi \,\partial_y\partial^{k-i-1}\rho \, dx dy
\end{align*}
and as before, we clearly have
\begin{align}\label{J2_A}
\sum_{i=1}^{k-1}J_2(i)&\leq  ||\partial^{k-1}\Delta\psi||_{L^2}\left[ \sum_{i=1}^{k-3} ||\partial^{i+1}\partial_x\psi||_{L^{\infty}} ||\partial_y\partial^{k-i-1}\rho||_{L^2}+\sum_{i=k-2}^{k-1} ||\partial^{i+1}\partial_x\psi||_{L^{2}} ||\partial_y\partial^{k-i-1}\rho||_{L^{\infty}} \right] \nonumber \\
&\leq ||\rho||_{H^{k}(\Omega)}\, ||\nabla\psi||_{H^{k}(\Omega)}^2 \qquad \text{for}\qquad k\geq 4.
\end{align}
	
$\hspace{-1.1 cm} \bullet$ \quad  All derivatives are in $y$. This is $\partial^{k}\equiv \partial_y^{k}$. In this case, we have that:
\begin{align*}
J_2(i)=\int_{\Omega} \partial_y^{k}\rho \, \partial_y^{i+1}\partial_x\psi \,\partial_y^{k-i}\rho \, dx dy
\end{align*}
and by integration by parts we achieve:
\begin{align*}
\hspace{-0.8 cm} J_2(i)&=\int_{\Omega}\partial_y^{k-1}\partial_x\rho \, \partial_y^{k-i+1}\rho\, \partial_y^{i+1}\psi\, dxdy+\int_{\Omega}\partial_y^{k-1}\partial_x\rho \, \partial_y^{k-i}\rho\, \partial_y^{i+2}\psi\, dxdy-\int_{\Omega}\partial_y^{k}\rho\, \partial_y^{k-i}\partial_x\rho \, \partial_y^{i+1}\psi \, dxdy\\
&-\int_{\Omega}\partial_y\left[\partial_y^{k-1}\partial_x\rho \, \partial_y^{k-i}\rho\, \partial_y^{i+1}\psi\right]\, dxdy +\int_{\Omega}\partial_x\left[\partial_y^{k} \rho\, \partial_y^{i+1} \psi \, \partial_y^{k-i} \rho \right] \, dxdy.
\end{align*}
By the periodicity in the horizontal varible, it is clear that the only boundary term that needs to be study  carefully is the first one, which vanish bucause $\rho\in X^k(\Omega)$ and $\psi\in X^{k+1}(\Omega)$.

Therefore, we get:
\begin{align*}
\hspace{-0.8 cm} J_2(i)&=\int_{\Omega}\partial_y^{k-1}\partial_x\rho \, \partial_y^{k-i+1}\rho\, \partial_y^{i+1}\psi\, dxdy+\int_{\Omega}\partial_y^{k-1}\partial_x\rho \, \partial_y^{k-i}\rho\, \partial_y^{i+2}\psi\, dxdy-\int_{\Omega}\partial_y^{k}\rho\, \partial_y^{k-i}\partial_x\rho \, \partial_y^{i+1}\psi \, dxdy\\
&=-\int_{\Omega}\partial_y^{k-1}\Delta\psi \, \partial_y^{k-i+1}\rho\, \partial_y^{i+1}\psi\, dxdy-\int_{\Omega}\partial_y^{k-1}\Delta\psi \, \partial_y^{k-i}\rho\, \partial_y^{i+2}\psi\, dxdy+\int_{\Omega}\partial_y^{k}\rho\, \partial_y^{k-i}\Delta\psi \, \partial_y^{i+1}\psi \, dxdy
\end{align*}
where in the last equality we have used (\ref{Poisson_Homogeneous_Neumann_psi}). Repeatedly applying H\"older's inequality we obtain that:
\begin{align*}
\sum_{i=1}^{k-1}J_2(i)&\leq ||\partial_y^{k-1}\Delta\psi||_{L^2}\left[\sum_{i=1}^{k-2} ||\partial_y^{k-i+1}\rho||_{L^2}\,||\partial_y^{i+1}\psi||_{L^{\infty}}+ ||\partial_y^{2}\rho||_{L^{\infty}}\,||\partial_y^{k}\psi||_{L^2}\right]\\
&+||\partial_y^{k-1}\Delta\psi||_{L^2}\left[\sum_{i=1}^{k-3} ||\partial_y^{k-i}\rho||_{L^2}\,||\partial_y^{i+2}\psi||_{L^{\infty}}+ \sum_{i=k-2}^{k-1} ||\partial_y^{k-i}\rho||_{L^{\infty}}\,||\partial_y^{i+2}\psi||_{L^{2}}\right]\\
&+||\partial_y^{k}\rho||_{L^2}\left[\sum_{i=1}^{k-2} ||\partial_y^{k-i}\Delta\psi||_{L^2}\, ||\partial_y^{i+1}\psi||_{L^{\infty}} + ||\partial_y\Delta\psi||_{L^{\infty}}\, ||\partial_y^{k}\psi||_{L^{2}}\right].
\end{align*}
Then, by the Sobolev embedding, we clearly have:
\begin{equation}\label{J2_B}
\sum_{i=1}^{k-1}J_2(i)\leq ||\rho||_{H^{k}(\Omega)}\, ||\nabla\psi||_{H^{k}(\Omega)}^2 \qquad \text{for}\qquad k\geq 6.
\end{equation}
Putting together (\ref{SPECIAL_TERM}), (\ref{J1}), (\ref{J2_A}) and (\ref{J2_B}) we have proved that:
\begin{equation}\label{I2}
I_2\lesssim ||\partial u_2||_{L^{\infty}(\Omega)}\, ||\rho||_{H^{k}(\Omega)}^2+||\rho||_{H^{k}(\Omega)}\, ||\nabla\psi||_{H^{k}(\Omega)}^2 \qquad \text{for}\qquad k\geq 6.
\end{equation}

\noindent
To sum up, we have obtained the next energy estimate.
\begin{thm}\label{Energy_Estimate}
Let $\rho(t)\in X^k(\Omega)$ be a solution of (\ref{System_A_PI}) for any $t\geq 0$. Then, the following estimate holds for \nolinebreak$k\geq 6$:
\begin{equation}\label{energy_estimate}
\tfrac{1}{2}\partial_t ||\rho||_{H^{k}(\Omega)}^2(t) \leq  C\,||\partial u_2||_{L^{\infty}(\Omega)}(t)\, ||\rho||_{H^{k}(\Omega)}^2(t)- \left(1-C\,||\rho||_{H^{k}(\Omega)}(t)\right)\, ||\mathbf{u}||_{H^{k}(\Omega)}^2(t).
\end{equation}
\end{thm}
\begin{proof}
First of all, we remember that $\mathbf{u}=\nabla^{\perp}\psi$ and:
\begin{align*}
\tfrac{1}{2}\partial_t ||\rho||_{L^2(\Omega)}^2\, &=-||\nabla\psi||_{L^2(\Omega)}^2,\\
\tfrac{1}{2}\partial_t ||\rho||_{\dot{H}^{k}(\Omega)}^2&=-||\nabla\psi||_{\dot{H}^{k}(\Omega)}^2+I_2,
\end{align*}
so summing and applying (\ref{I2}), we have achieved our target.
\end{proof}

As we want to prove a global existence in time result for small data, this is $||\rho||_{H^{k}(\Omega)}(t)<< 1$. Then, the second term in the energy estimate (\ref{energy_estimate}) is a ``good'' one, because it has the right sign. In consequence, we fix our attention in the first term. If we have a ``good'' time decay of the $L^{\infty}(\Omega)$-norm of $\partial u_2$, then we will be able to prove that $||\rho||_{H^{k}(\Omega)}(t)$ remains small for all time by a boostraping argument.

\subsection{Linear \& Non-Linear estimates}\label{Sec_5.2}
\noindent
Our goal for the rest of the paper is to obtain time decay estimates for $||\partial u_2||_{L^{\infty}(\Omega)}(t).$
As we will see in the next Section \ref{Sec_5.3}, to close the energy estimate and finish the proof is enough to get an integrable rate.\\

We approach the question of global well-posedness for a small initial data from a perturbative point of view, i.e., we see (\ref{System_B}) as a non-linear perturbation of the linear problem. The linearized system of (\ref{System_B}) around the trivial solution $(\rho,\mathbf{u})=(0,0)$ reads
\begin{align*}
\left\{
\begin{array}{rl}
\partial_t\bar{\rho}  &=  u_2  \\
\partial_t\tilde{\rho} &=0\\
\mathbf{u}&=-\nabla\Pi-(0,\bar{\rho})\\
\nabla\cdot\mathbf{u}&=0
\end{array}
\right.
\end{align*}
together with the no-slip condition $\mathbf{u}\cdot \mathbf{n}=0$ on $\partial\Omega$  and initial data $\rho(0)\in X^k(\Omega)$ such that $\rho(0)=\bar{\rho}(0)+\tilde{\rho}(0)$.

It is not difficult to prove that $\bar{\rho}$ will decay in time and $\tilde{\rho}$ will just remain bounded at linear order. Consequently, the linearized problem has a very large set of stationary (undamped) modes.

Now, we return to our non-linear problem:
\begin{equation*}
\left\{
\begin{array}{rl}
\partial_t \bar{\rho}+\overline{\textbf{u}\cdot \nabla \bar{\rho}}+\partial_y\tilde{\rho} \,u_2&=u_2\\
\partial_t \tilde{\rho}+ \widetilde{\textbf{u}\cdot \nabla \bar{\rho}}\hspace{1.4 cm }&=0\\
\textbf{u}&=-\nabla\Pi-(0,\bar{\rho})\\
\nabla\cdot\mathbf{u}&=0
\end{array}
\right.
\end{equation*}
together with the no-slip condition $\mathbf{u}\cdot \mathbf{n}=0$ on $\partial\Omega$. Since $\bar{\rho}$ is decaying, the term $\overline{\mathbf{u}\cdot\nabla\bar{\rho}}$ should be small and controllable. The term $\partial_y\tilde{\rho}\, u_2$, however, acts like a second linear operator since $\tilde{\rho}$ is not decaying. It is conceivable that this extra quasi-linear operator could compete with the damping coming from the linear term. This makes the problem of long-time behavior more difficult.

As in \cite{Elgindi} we solve this by, more or less, doing a second linearization around the undamped modes and showing that the stationary modes can be controlled. Then, we wish to prove decay estimates for $\bar{\rho}$ in the following system:
\begin{equation*}
\left\{
\begin{array}{rl}
\partial_t \bar{\rho}&=(1-\partial_y\tilde{\rho}) \,u_2\\
\partial_t \tilde{\rho}&=0\\
\textbf{u}&=-\nabla\Pi-(0,\bar{\rho})\\
\nabla\cdot\mathbf{u}&=0
\end{array}
\right.
\end{equation*}
assuming that the initial data $\rho(0)=\bar{\rho}(0)+\tilde{\rho}(0)$ is sufficiently small. By showing that, we find the decay mechanism is ``stable'' with respect to the sort of perturbations which
this second linear operator introduces, and we are able to keep the decay mechanism
and close a decay estimate for $\bar{\rho}$ and show that $\tilde{\rho}$, while not decaying, converges as $t\to\infty$.\\

We notice that the second equation $\partial_t\tilde{\rho}(t)=0$ reduces to a condition at time $t=0$, i.e. $\tilde{\rho}(y,t)=\tilde{\rho}(y,0)$. In consequence $\tilde{\rho}$ will just remain bounded and our goal is to solve the following system in $\Omega$:
\begin{align}\label{Lineal_u2}
\left\{
\begin{array}{rl}
\partial_t\bar{\rho}  &=(1-\partial_y\tilde{\rho}) \,  u_2  \\
\mathbf{u}&=-\nabla\Pi-(0,\bar{\rho})\\
\nabla\cdot\mathbf{u}&=0
\end{array}
\right.
\end{align}
besides the no-slip  condition $\mathbf{u}\cdot \mathbf{n}=0$ on $\partial\Omega$. Using the \textit{stream formulation} (\ref{Stream_1}), we can rewrite (\ref{Lineal_u2}) in a  more adequate way as:
\begin{equation}\label{Linear_Psi}
\left\{
\begin{array}{ll}
\partial_t\bar{\rho}&=(1-\partial_y\tilde{\rho}) \,\partial_x\psi\\
\bar{\rho}|_{t=0}&=\bar{\rho}(0)
\end{array}
\right.
\end{equation}
where $\psi$ is the solution of the Poisson problem (\ref{Poisson_Homogeneous_Neumann_psi}) and $\rho(0)\in X^k(\Omega).$

\subsubsection{Quasi-Linear Decay}
In this subsection we prove $L^2(\Omega)$ decay estimates for the quasi-linear equation:
\begin{equation}\label{new_lineal}
\partial_t\bar{\rho}=(1-G(y,t)) \,\partial_x\psi
\end{equation}
where $\psi$ is the  solution of (\ref{Poisson_Homogeneous_Neumann_psi}) given by (\ref{expresion_psi}) and $G(y,t)$, which plays the role of $\partial_y\tilde{\rho}(y,t)$, is sufficiently small.\\

\noindent
\textbf{Remark:} We cannot extract  an exact formula for the solution by taking the analog of Fourier transform given by the eigenfunction expansion, because the $G(y,t)$ term mixes the effects of all the Fourier coefficients.
\begin{lemma} There exists $\varepsilon>0$ small enough such that if $||G||_{H^2([-1,1])}(t)\leq \varepsilon$ for all time, then the solution of equation (\ref{new_lineal}) satisfies that:
$$\partial_t||\bar{\rho}||_{L^2(\Omega)}^2(t)\lesssim -||\nabla \psi||_{L^2(\Omega)}^2(t)$$
where $\psi$ is the  solution of (\ref{Poisson_Homogeneous_Neumann_psi}).
\end{lemma}
\begin{proof}
Upon multiplying (\ref{new_lineal}) by $\bar{\rho}$ and integrating we see:
\begin{align*}
\tfrac{1}{2}\partial_t ||\bar{\rho}||_{L^2(\Omega)}^2=\int_{\Omega} \left(1-G(y)\right) \partial_x\psi \bar{\rho}\,dx dy.
\end{align*}
After integrating by parts and use the \textit{stream function} $\psi$, which is solution of the Poisson's problem (\ref{Poisson_Homogeneous_Neumann_psi}), we arrive at:
\begin{align*}
\tfrac{1}{2}\partial_t ||\bar{\rho}||_{L^2(\Omega)}^2&=\int_{\Omega} \left(1-G(y)\right) \,\psi \Delta\psi\,dx dy\\
&=-\int_{\Omega} \left(1-G(y)\right) \, |\nabla \psi|^2\,dxdy+\int_{\Omega}G'(y)\,\psi\,\partial_y\psi\,dxdy.
\end{align*}
Now, applying the Sobolev embedding $L^{\infty}([-1,1])\hookrightarrow H^1([-1,1])$  and the Poincar\'e inequality we get:
$$\tfrac{1}{2}\partial_t ||\bar{\rho}||_{L^2(\Omega)}^2(t)\leq -\left[1-C\,||G||_{H^2([-1,1])}(t)\right]\,||\nabla\psi||_{L^2(\Omega)}^2(t).$$
Then, as $||G||_{H^2([-1,1])}(t)$ is small enough for all time, we get that $||\bar{\rho}||_{L^2(\Omega)}(t)$ is bounded by its initial \nolinebreak data.
\end{proof}

As in \cite{Elgindi}, due to the fact that the Laplacian has discrete spectrum on $\Omega$ we can actually deduce that $\bar{\rho}$ decays in $L^2(\Omega)$ so long as its higher derivatives are controlled.

\begin{lemma}Let $\alpha\in\N$ and $N:\R^{+}\longrightarrow \R^{+}$.
The solution of (\ref{Poisson_Homogeneous_Neumann_psi}) satisfies the following lower bound:
\begin{equation}{\label{lower_bound}}
||\nabla\psi||_{L^2(\Omega)}^2(t)\geq \frac{1}{N(t)}||\bar{\rho}||_{L^2(\Omega)}^2(t)-\frac{1}{N(t)^{1+\alpha}}||\bar{\rho}||_{H^{\alpha}(\Omega)}^2(t)
\end{equation}
\end{lemma}
\begin{proof} The solution of (\ref{Poisson_Homogeneous_Neumann_psi}) is given by:
\begin{equation*}
\psi(x,y)=\, \sum_{p\in\Z}\sum_{q\in\N}\left( \frac{ip}{p^2+ \left(q\tfrac{\pi}{2}\right)^2 }\right)\, \mathcal{F}_{\omega}[\bar{\rho}](p,q)\, \omega_{p,q}(x,y).
\end{equation*}
Moreover, as $||\nabla \psi ||_{L^2(\Omega)}^2(t)=-(\psi,\Delta\psi)=(\psi,\partial_x\bar{\rho})$, it is clear that:
\begin{align*}
||\nabla \psi ||_{L^2(\Omega)}^2(t)&=\sum_{p\in\Z}\sum_{q\in\N}\left( \frac{p^2}{p^2+ \left(q\tfrac{\pi}{2}\right)^2 }\right)^2\, \big|\mathcal{F}_{\omega}[\bar{\rho}](p,q)\big|^2.
\end{align*}
Now, on one hand, we introduce the auxiliary function $N:\R^{+}\longrightarrow \R^{+}$ to obtain that:
\begin{align}\label{part_1_lemma}
||\nabla \psi ||_{L^2(\Omega)}^2(t)&\geq \frac{1}{N(t)}||\bar{\rho}||_{L^2(\Omega)}^2(t)+\sum_{(p,q)\in\Z_{\neq 0}\times\N} \left(\frac{1}{p^2+ \left(q\tfrac{\pi}{2}\right)^2 }-\frac{1}{N(t)}\right)\, \big|\mathcal{F}_{\omega}[\bar{\rho}](p,q)\big|^2\nonumber\\
&\geq \frac{1}{N(t)}\left(||\bar{\rho}||_{L^2(\Omega)}^2(t)-\sum_{p^2+\left(q\tfrac{\pi}{2}\right)^2\geq N(t)} \big|\mathcal{F}_{\omega}[\bar{\rho}](p,q)\big|^2 \right)
\end{align}
On the other hand, by Corollary (\ref{thm_norma_expansion}) we have that:
\begin{align}\label{part_2_lemma}
\sum_{p^2+\left(q\tfrac{\pi}{2}\right)^2\geq N(t)} \big|\mathcal{F}_{\omega}[\bar{\rho}](p,q)\big|^2 &\leq \frac{1}{N(t)^{\alpha}}\sum_{p^2+\left(q\tfrac{\pi}{2}\right)^2\geq N(t)}\left(p^2+\left(q\tfrac{\pi}{2}\right)^2\right)^{\alpha} \big|\mathcal{F}_{\omega}[\bar{\rho}](p,q)\big|^2\nonumber\\
&\leq \frac{1}{N(t)^{\alpha}} ||\bar{\rho}||_{H^{\alpha}(\Omega)}^2(t)
\end{align}
Combining the preceding estimates (\ref{part_1_lemma}) and (\ref{part_2_lemma}) we arrive at (\ref{lower_bound}).
\end{proof}

This gives that:
\begin{equation*}
\partial_t||\bar{\rho}||_{L^2(\Omega)}^2(t)\lesssim -\frac{1}{N(t)}||\bar{\rho}||_{L^2(\Omega)}^2(t)+\frac{1}{N(t)^{1+\alpha}}||\bar{\rho}||_{H^{\alpha}(\Omega)}^2(t)
\end{equation*}
and assuming that $N:\R^{+}\longrightarrow \R^{+}$ satisfies that $N'(t)N(t)\geq 1$ we obtain:
\begin{equation}\label{lema_decaimiento}
||\bar{\rho}||_{L^2(\Omega)}^2(t)\lesssim e^{-(N(t)-N(0))}||\bar{\rho}||_{L^2(\Omega)}^2(0)+\int_{0}^{t}\frac{e^{-(N(t)-N(s))}}{N(s)^{1+\alpha}}\,||\bar{\rho}||_{H^{\alpha}(\Omega)}^2(s)\, ds.
\end{equation}
For simplicity, we take $N(t):=2\sqrt{1+t}$ in (\ref{lema_decaimiento}), which give us:
$$||\bar{\rho}||_{L^2(\Omega)}^2(t)\lesssim e^{-2\sqrt{1+t}}||\bar{\rho}||_{L^2(\Omega)}^2(0)+\left(\int_{0}^{t}\frac{e^{-2(\sqrt{1+t}-\sqrt{1+s})}}{(1+s)^{\tfrac{1+\alpha}{2}}}\, ds\right)||\bar{\rho}||_{L^{\infty}([0,t],H^{\alpha}(\Omega))}^2.$$
Now, we use the following calculus lemma.
\begin{lemma}Let $\alpha\in \N$, we have that:
$$\int_{0}^{t}\frac{e^{-2(\sqrt{1+t}-\sqrt{1+s})}}{(1+s)^{\tfrac{1+\alpha}{2}}}\, ds \lesssim \frac{1}{(1+t)^{\tfrac{\alpha}{2}}}$$
\end{lemma}
\begin{proof}
The proof of this lemma is simple and basically follows after we split the integral into two pieces: one from 0 to $t/2$ and the other from $t/2$ to $t$. The integral from $0$ to $t/2$ decays exponentially. The second integral decays like $(1+t)^{-\frac{1+\alpha}{2}}$ multiplied by the following factor:
$$\int_{t/2}^{t}e^{-2(\sqrt{1+t}-\sqrt{1+s})}\, ds=\int_{0}^{2\left(\sqrt{1+t}-\sqrt{1+t/2}\right)}e^{-\tau}\left(\sqrt{1+t}-\frac{\tau}{2}\right)\,d\tau\lesssim \sqrt{1+t}$$
This completes the proof.
\end{proof}
Then, applying the previous lemma we see that:
\begin{equation}\label{decay_l2_tarek}
||\bar{\rho}||_{L^2(\Omega)}^2(t)\lesssim \frac{||\bar{\rho}||_{L^{\infty}([0,t],H^{\alpha}(\Omega))}^2}{(1+t)^{\tfrac{\alpha}{2}}}.
\end{equation}
Now, we wish to prove a similar decay estimate for the higher derivatives. The idea is then to show that $||\bar{\rho}||_{H^{\alpha}(\Omega)}^2(t)$ is bounded by its initial data and then this would give (\ref{decay_l2_tarek}) with $L^2(\Omega)$ replaced by $H^{k}(\Omega)$ and $H^{\alpha}(\Omega)$ replaced by $H^{k+\alpha}(\Omega)$.
\begin{lemma} Let $k\in \N\cup\{0\}$ and fix an auxiliary parameter $\alpha\in \N$.
There exists $\varepsilon>0$ small enough such that if $||G||_{H^{k+\alpha+2}([-1,1])}(t)\leq \varepsilon$ for all time and $\bar{\rho}(0)\in H^{k+\alpha}(\Omega)$, then the solution of equation (\ref{new_lineal}) satisfies:
$$||\bar{\rho}||_{H^{k}(\Omega)}^2(t)\lesssim \frac{||\bar{\rho}||_{H^{k+\alpha}(\Omega)}^2(0)}{(1+t)^{\tfrac{\alpha}{2}}}.$$
\end{lemma}
\begin{proof}Fix $n\in\N\cup\{0\}$ such that $n\leq k+\alpha$. First, we will prove that $||\bar{\rho}||_{H^n(\Omega)}^2(t)\leq ||\bar{\rho}||_{H^n(\Omega)}^2(0)$. Proceeding as before, after integrating by parts and using the \textit{stream function} $\psi$, we arrive at:
\begin{align*}
\tfrac{1}{2}\partial_t||\bar{\rho}||_{H^n(\Omega)}^2=\int_{\Omega}\partial^n\left[(1-G(y))\psi\right]\partial^n\Delta\psi\,dxdy.
\end{align*}
By Leibniz's rule we have that:
\begin{align*}
\tfrac{1}{2}\partial_t||\bar{\rho}||_{H^n(\Omega)}^2=\int_{\Omega}(1-G(y))\partial^n\psi \partial^n\Delta\psi\,dxdy +\sum_{i=1}^{n} \binom{n}{i} \int_{\Omega}\partial^{i}(1-G(y))\partial^{n-i}\psi\, \partial^n\Delta\psi\,dxdy
\end{align*}
As before, applying the Sobolev embedding $L^{\infty}([-1,1])\hookrightarrow H^1([-1,1])$  and the Poincar\'e inequality we get:
$$\tfrac{1}{2}\partial_t ||\bar{\rho}||_{H^n(\Omega)}^2(t)\leq -\left[1-C\,||G||_{H^{n+2}([-1,1])}(t)\right]\,||\nabla\psi||_{H^n(\Omega)}^2(t).$$
Then, as $||G||_{H^{n+2}([-1,1])}(t)$ is small enough for all time, we get that $||\bar{\rho}||_{H^n(\Omega)}(t)$ is bounded by its initial \nolinebreak data.
Applying this in (\ref{decay_l2_tarek}), we have proved our goal for the case $k=0$. Arguing as we did above when we proved the $L^2(\Omega)\equiv H^{0}(\Omega)$ decay, we can extend the result for general $k\in\N$.
\end{proof}

\subsubsection{Non-Linear Decay}
Next, we will show how this decay of the quasi-linear solutions can be used to establish the stability of the stationary solution $(\rho,\mathbf{u})=(0,0)$ for the general problem (\ref{System_B}). When perturbing around the stationary solution, we get the following system:

\begin{equation}\label{Duhamel}
\begin{cases}
\partial_t\bar{\rho}-(1-\partial_y\tilde{\rho}) u_2&=\quad -\overline{\mathbf{u}\cdot\nabla\bar{\rho}}\\
\partial_t \tilde{\rho}&=\quad -\widetilde{\mathbf{u}\cdot\nabla\bar{\rho}}
\end{cases}
\end{equation}
where $\mathbf{u}=\nabla^{\perp}\psi$ and $\psi$ is the solution of (\ref{Poisson_Homogeneous_Neumann_psi}).

Using Duhamel's formula, with $G(y,t)\equiv \partial_y\tilde{\rho}(y,t)$ small enough in the adequate norm, we write the solution of (\ref{Duhamel}) as:
$$\bar{\rho}(t)=e^{\mathscr{L}(t,0)}\bar{\rho}(0)-\int_{0}^{t}e^{\mathscr{L}(t,s)}\left[\overline{\mathbf{u}\cdot\nabla\bar{\rho}}\right](s)\,ds \qquad \text{and} \qquad \tilde{\rho}(t)= \tilde{\rho}(0)-\int_{0}^{t}\widetilde{\mathbf{u}\cdot\nabla\bar{\rho}}(s)\,ds$$
where $e^{\mathscr{L}(t,s)}$ denotes  the solution operator of the associated quasi-linear problem (\ref{Linear_Psi}) from $s$ to $t$.
Therefore, we have:
$$||\bar{\rho}||_{H^{n}(\Omega)}(t)\lesssim \frac{||\bar{\rho}||_{H^{n+\alpha}(\Omega)}(0)}{(1+t)^{\frac{\alpha}{4}}}+\int_{0}^{t}\frac{1}{(1+(t-s))^{\frac{\alpha}{4}}}\, ||\overline{\mathbf{u}\cdot\nabla\bar{\rho}}||_{H^{n+\alpha}(\Omega)}(s)\,ds.\\$$

\subsection{The Bootstraping}\label{Sec_5.3}
We now demonstrate the bootstrap argument used to prove our goal. The general approach here is a
typical continuity argument that has been used successfully in a plethora of other cases. Theorem (\ref{Energy_Estimate}) tell us that the following estimate holds for $k\geq 6$:
\begin{equation}\label{main_estimate}
\tfrac{1}{2}\partial_t ||\rho||_{H^{k}(\Omega)}^2(t)\leq C\,||\partial u_2||_{L^{\infty}(\Omega)}(t)\, ||\rho||_{H^{k}(\Omega)}^2(t)- \left(1-C\,||\rho||_{H^{k}(\Omega)}(t)\right)\, ||\mathbf{u}||_{H^{k}(\Omega)}^2(t).
\end{equation}

\noindent

\noindent
We need to prove:
\begin{lemma}\label{bootstrap_lemma}
If $||\rho||_{H^{\kappa}(\Omega)}(0)<\varepsilon$ and $||\rho||_{H^{\kappa}(\Omega)}(t)\leq 4\,\varepsilon$ on the interval $[0,T]$ with $0<\varepsilon\leq \nolinebreak \varepsilon_0$ small enough. Then $||\rho||_{H^{\kappa}(\Omega)}(t)$ remains uniformly bounded by $2\,\varepsilon$ on the same interval $[0,T]$.
\end{lemma}

We will prove Lemma (\ref{bootstrap_lemma}) through a bootstrap argument, where the main ingredient is the estimate (\ref{main_estimate}). We will work with the following bootstrap hypothesis, to assume that $||\rho||_{H^{\kappa}(\Omega)}(t)\leq 4\varepsilon$ on the interval $[0,T]$, where $\kappa$ is big enought and $0<\varepsilon<<1$ such that:
 $$\left(1-C\,||\rho||_{H^{\kappa}(\Omega)}(t)\right)\geq 0 \quad  \text{on} \quad  [0,T].$$

Then, by Gr\"onwall's inequality we have:
$$||\rho||_{H^{\kappa}(\Omega)}(t)\leq ||\rho||_{H^{\kappa}(\Omega)}(0)\, \exp\left(C \int_{0}^{t}||\partial u_2||_{L^{\infty}(\Omega)}(s)\, ds\right) \qquad t\in[0,T].$$
Our goal is to prove that $||\partial u_2||_{L^{\infty}(\Omega)}(t)$ decays on time at an integrable rate. As $L^{\infty}(\Omega)\hookrightarrow H^{2}(\Omega)$ by the Sobolev embedding, it is enough to prove it for $||u_2||_{H^{3}(\Omega)}(t)$. This will allow us to close the energy estimate and finish the proof.

\subsubsection{Integral decay of $||u_2||_{H^{3}(\Omega)}$}
In order to control $||u_2||_{H^{3}(\Omega)}$ in time it is enough to control $||\bar{\rho}||_{H^3(\Omega)}$. We have the following result.
\begin{lemma} Assume that $||\rho||_{H^{\kappa}(\Omega)}(t)\leq 4\,\varepsilon$  for all $t\in[0,T]$ where $\kappa\geq 5+2\gamma$ with $\gamma>4$. Then
$$||\bar{\rho}||_{H^{3}(\Omega)}(t)\lesssim \frac{\varepsilon}{(1+t)^{\frac{\gamma}{4}}} \qquad \text{for all}\quad  t\in[0,T].$$
\end{lemma}
\begin{proof} By assumption $\partial_y\tilde{\rho}(t)$ is small in $H^{\kappa-1}(\Omega)$ for all $t\in[0,T]$. This implies that $e^{\mathscr{L}(t,s)}$ has nice decay properties for $s\leq t$ and $t\in[0,T]$ in $H^{3}(\Omega)$ if $\kappa\geq 6+\gamma$. Hence, Duhamel's formula give us:
$$||\bar{\rho}||_{H^{3}(\Omega)}(t)\lesssim \frac{||\bar{\rho}||_{H^{3+\gamma}(\Omega)}(0)}{(1+t)^{\frac{\gamma}{4}}}+\int_{0}^{t}\frac{1}{(1+(t-s))^{\frac{\gamma}{4}}}\, ||\overline{\mathbf{u}\cdot\nabla\bar{\rho}}||_{H^{3+ \gamma}(\Omega)}(s)\,ds$$
and we have that:
\begin{align*}
||\overline{\mathbf{u}\cdot\nabla\bar{\rho}}||_{H^{3+\gamma}}&\leq ||\mathbf{u}\cdot\nabla\bar{\rho}||_{H^{3+\gamma}}\lesssim ||\mathbf{u}||_{H^{3+\gamma}(\Omega)}\,||\bar{\rho}||_{H^{4+\gamma}(\Omega)}\lesssim ||\bar{\rho}||_{H^{4+\gamma}(\Omega)}^2.
\end{align*}
Hence
$$||\bar{\rho}||_{H^{3}(\Omega)}(t)\lesssim \frac{||\bar{\rho}||_{H^{3+\gamma}(\Omega)}(0)}{(1+t)^{\frac{\gamma}{4}}}+\int_{0}^{t}\frac{1}{(1+(t-s))^{\frac{\gamma}{4}}}\, ||\bar{\rho}||_{H^{4+ \gamma}(\Omega)}^2(s)\,ds$$
and, in conclusion, we need a control in time of $||\bar{\rho}||_{H^{4+\gamma}(\Omega)}$.

However, due to the well-known Gagliardo-Nirenberg interpolation inequalities:
$$||D^{j}f||_{L^2(\Omega)}\leq C\, ||D^{2\,j}f||_{L^2(\Omega)}^{1/2}\, ||f||_{L^2(\Omega)}^{1/2}+\tilde{C}\,||f||_{L^2(\Omega)}$$
we get
\begin{equation}\label{GN_rho}
||\bar{\rho}||_{H^{4+\gamma}(\Omega)}\lesssim  ||\bar{\rho}||_{H^{3+2\left(1+\gamma\right)}(\Omega)}^{1/2}\, ||\bar{\rho}||_{H^{3}(\Omega)}^{1/2}
\end{equation}
\noindent
Therefore, if we apply (\ref{GN_rho}) in the previous inequalities, we get:
\begin{align*}
\hspace{-0.5 cm}||\bar{\rho}||_{H^{3}(\Omega)}(t)&\lesssim \frac{||\bar{\rho}||_{H^{3+\gamma}(\Omega)}(0)}{(1+t)^{\frac{\gamma}{4}}}+\int_{0}^{t}\frac{ ||\bar{\rho}||_{H^{\kappa}(\Omega)}(s)}{(1+(t-s))^{\frac{\gamma}{4}}}\,||\bar{\rho}||_{H^{3}(\Omega)}(s)\,ds
\end{align*}
where, we have defined $\kappa\in\N$ so that $\kappa\geq \max\{5+2\gamma,6+\gamma\}$.\\

\noindent
By hypothesis, we have that $||\rho||_{H^{\kappa}(\Omega)}(t)\leq 4\varepsilon$ on the interval $[0,T]$. Then, we obtain that:
\begin{align*}
\hspace{-0.5 cm}||\bar{\rho}||_{H^{3}(\Omega)}(t)&\leq \frac{C\varepsilon}{(1+t)^{\frac{\gamma}{4}}}+\int_{0}^{t}\frac{C\varepsilon}{(1+(t-s))^{\frac{\gamma}{4}}}\,||\bar{\rho}||_{H^{3}(\Omega)}(s)\,ds
\end{align*}
\noindent
In particular, there exist $0<T^{\star}(C)\leq T$ such that for $t\in[0,T^{\star}(C)]$ we have:
$$||\bar{\rho}||_{H^{3}}(t)\leq 4 \,\frac{C\,\varepsilon}{(1+t)^{\frac{\gamma}{4}}}.$$

\noindent
The following basic  lemma is stated without proof (for a proof see \cite[p.~584]{Elgindi}).
\begin{lemma}\label{Basic_Lemma}
Let $\delta,q>0$, then:
$$\int_{0}^{t}\frac{ds}{(1+(t-s))^{\delta}\,(1+s)^{1+q}}\leq \frac{\mathcal{C}_{\delta,q}}{(1+t)^{\min\{\delta,1+q\}}}$$
\end{lemma}
\noindent
If we restrict to $0\leq t\leq T^{\star}(C)$ and we apply  the previous Lemma (\ref{Basic_Lemma}), we have:
\begin{align*}
||\bar{\rho}||_{H^{3}(\Omega)}(t)&\leq \frac{C\,\varepsilon}{(1+t)^{\frac{\gamma}{4}}}+\int_{0}^{t}\frac{ C\,\varepsilon}{(1+(t-s))^{\frac{\gamma}{4}}}\,\frac{4\,C\,\varepsilon}{(1+s)^{\frac{\gamma}{4}}}\,ds\\
&\leq \frac{C\,\varepsilon}{(1+t)^{\frac{\gamma}{4}}}+\frac{4\,\tilde{C}\,\varepsilon^2}{(1+t)^{\frac{\gamma}{4}}}
\end{align*}
The last term in the expression above is quadratic in $\varepsilon$, it is enough to find $0<\epsilon<<1$ small enough so that
$$||\bar{\rho}||_{H^{3}(\Omega)}(t)\leq 2 \,\frac{C\,\varepsilon}{(1+t)^{\frac{\gamma}{4}}}$$
for all $t\in[0,T^{\star}(C)]$ and, by continuity, for all $t\in[0,T]$.\\
\end{proof}


\noindent
So, with $\gamma>4$ we have proved the integrable decay of $u_2$, then we are able to close our energy estimate.
We are now in the position to show how the bootstrap can be closed. This is merely a matter of collecting the conditions established above and showing that they can indeed be satisfied.

In conclusion, if $||\rho||_{H^{\kappa}(\Omega)}(t)\leq 4\,\varepsilon$ for all $t\in[0,T]$ we have that
\begin{align*}
||\rho||_{H^{\kappa}(\Omega)}(t)&\leq ||\rho||_{H^{\kappa}(\Omega)}(0)\, \exp\left(C\int_{0}^{t}||\partial u_2||_{L^{\infty}(\Omega)}(s)\, ds\right) \\
&\leq\varepsilon\,\exp \left(C\int_{0}^{t}\frac{\tilde{C}\varepsilon}{(1+s)^{\frac{\gamma}{4}}}\,ds\right)\leq \varepsilon \exp\left(C^{\star}\varepsilon\right)
\end{align*}
and $||\rho||_{H^{\kappa}(\Omega)}(t)\leq 2\,\varepsilon$ for all $t\in[0,T]$ if we consider $\varepsilon$ small enough, which allows us to prolong the solution and then repeat the argument for all time.\\

\noindent
\textbf{Funding:} The authors are supported by Spanish National Research Project MTM2014-59488-P and ICMAT Severo Ochoa projects SEV-2011-0087 and SEV-2015-556.  AC was partially supported by the ERC grant 307179-GFTIPFD, DC was partially supported by the ERC Advanced Grant 788250 and DL was supported by La Caixa-Severo Ochoa grant.\\

\noindent
\textbf{Acknowledgements:} The authors acknowledges helpful conversations with Tarek M. Elgindi.
The authors would like to thank the anonymous referee for their insightful comments and suggestions. \\

\noindent
\textbf{Conflict of Interest}: The authors declare that they have no conflict of interest.

\bibliography{bibliografia}
\bibliographystyle{plain}

\Addresses

\end{document}